\documentclass[11pt]{article}
\usepackage[utf8]{inputenc}

\usepackage{graphicx}
\usepackage{xcolor}
\usepackage{amsmath,amsthm,amssymb}
\usepackage[margin=1in]{geometry}
\usepackage[hyperindex,breaklinks]{hyperref}
\usepackage{comment}
\usepackage[normalem]{ulem}
\usepackage{cite}
\usepackage{tcolorbox}
\usepackage{caption}
\usepackage{subcaption}

\newcommand{\R}{\mathcal{R}}
\newcommand{\T}{{\tt T}}
\newcommand{\E}{\mathbb{E}}
\newcommand{\F}{\mathbb{F}}

\def\P{{\mathcal P}}
\def\ba{\begin{array}}
\def\ea{\end{array}}
\def\ds{\displaystyle}

\def\0{{\bf 0}}
\def\a{{\bf a}}
\def\b{{\bf b}}
\def\c{{\bf c}}
\def\C{{\bf C}}
\def\e{{\bf e}}

\def\x{{\bf x}}
\def\y{{\bf y}}

\def\v{{\bf v}}
\def\w{{\bf w}}
\def\u{{\bf u}}

\def\E{{\mathbb E}}
\def\S{{\mathcal S}}
\def\H{{\mathcal H}}
\def\C{{\mathcal C}}
\def\T{{\tt T}}

\def\TT{{\tt T}}
\def\F{{\tt F}}

\newcommand{\ket}[1]{| #1 \rangle}
\newcommand{\bra}[1]{\langle #1|}

\newcommand{\be}{\begin{equation}}
\newcommand{\ee}{\end{equation}}
\newcommand{\bea}{\begin{eqnarray}}
\newcommand{\eea}{\end{eqnarray}}
\newcommand{\bes}{\begin{equation*}}
\newcommand{\ees}{\end{equation*}}
\newcommand{\beas}{\begin{eqnarray*}}
\newcommand{\eeas}{\end{eqnarray*}}

\makeatletter
\newtheorem*{rep@theorem}{\rep@title}
\newcommand{\newreptheorem}[2]{%
\newenvironment{rep#1}[1]{%
 \def\rep@title{#2 \ref{##1} (restated)}%
 \begin{rep@theorem}}%
 {\end{rep@theorem}}}
\makeatother

\allowdisplaybreaks

\newtheorem{thm}{Theorem}[section]
\newtheorem*{thm*}{Theorem}
\newtheorem{cor}[thm]{Corollary}

\newtheorem{lem}[thm]{Lemma}
\newtheorem*{lem*}{Lemma}

\newtheorem{prop}[thm]{Proposition}
\newtheorem{defn}[thm]{Definition}
\newtheorem{rem}[thm]{Remark}

\newtheorem{fact}[thm]{Fact}

\newreptheorem{thm}{Theorem}
\newreptheorem{lem}{Lemma}
\newreptheorem{cor}{Corollary}

\numberwithin{equation}{section}
\usepackage{pgfplots}

\usepackage{authblk}

\title{ Reflective block Kaczmarz algorithms for least squares }

\author{Changpeng Shao\thanks{changpeng.shao@amss.ac.cn}}
\affil{Academy of Mathematics and Systems Science \\
Chinese Academy of Sciences \\
Beijing, 100190 China }

\date{\today}

\begin{document}

\maketitle

\begin{abstract}
In 
[\href{https://doi.org/10.1090/qam/1587 }{Steinerberger, Q. Appl. Math., 79:3, 419-429, 2021}] 
and 
[\href{https://epubs.siam.org/doi/abs/10.1137/21M1463306}{Shao, SIAM J. Matrix Anal. Appl. 44(1), 212-239, 2023}], two new types of Kaczmarz algorithms, which share some similarities, for consistent linear systems were proposed. These two algorithms not only compete with many previous Kaczmarz algorithms but, more importantly, reveal some interesting new geometric properties of solutions to linear systems that are not obvious from the standard viewpoint of the Kaczmarz algorithm. In this paper, we comprehensively study these two algorithms. First, we theoretically analyse the algorithms given in [\href{https://doi.org/10.1090/qam/1587}{Steinerberger, Q. Appl. Math., 79:3, 419-429, 2021}] for solving least squares. Second, we extend the two algorithms to block versions and provide their theoretical convergence rates. Our numerical experiments also verify the efficiency of these algorithms. Third, as a theoretical complement, we address some key questions left unanswered in [\href{https://epubs.siam.org/doi/abs/10.1137/21M1463306}{Shao, SIAM J. Matrix Anal. Appl. 44(1), 212-239, 2023}].
\end{abstract}

\vspace{.2cm}


\section{Introduction}

The Kaczmarz method is a useful iterative algorithm for solving linear systems of equations \cite{karczmarz1937angenaherte}. Due to its simple iterative scheme, various variants of the Kaczmarz method have been studied in the past, e.g., \cite{elfving1980block, benzi2004gianfranco,Cimmino,moorman2021randomized,needell2010randomized,needell2014paved,strohmer2009randomized,zouzias2013randomized,necoara2019faster,gower2015randomized,needell2014stochastic,han2024randomized,steinerberger2021surrounding,shao2021deterministic,needell2015randomized,ma2015convergence,liu2016accelerated}. In this paper, we focus on two new types of Kaczmarz algorithms proposed recently in \cite{steinerberger2021surrounding,shao2021deterministic}.

Let $A$ be an $m\times n$ real matrix and $\b$ be an $m\times 1$ real vector. We use $A_i$ to denote the $i$-th row of $A$. The Kaczmarz algorithm for solving the linear system $A\x=\b$, as studied in \cite{steinerberger2021surrounding,shao2021deterministic}, reads as follows: arbitrarily choose $\x_0\in \mathbb{R}^n$, update
\be
\label{intro:eq1}
\x_{k+1} = \x_k + 2 \frac{b_{i_k} - A_{i_k}^{\T} \x_k}{\|A_{i_k}\|^2} A_{i_k}, \quad k=0,1,2,\ldots,
\ee
where in \cite{steinerberger2021surrounding}, $i_k\in[m]$ is chosen with probability $\|A_{i_k}\|^2/\|A\|_\F^2$, while in \cite{shao2021deterministic}, $i_k = (k\mod m)+1$. Here $\|A\|_\F$ is the Frobenius norm of $A$.
So, the former is randomised, and the latter is deterministic. As shown in \cite{steinerberger2021surrounding,shao2021deterministic}, these two new types of Kaczmarz algorithms can outperform previous ones, such as the randomised Kaczmarz algorithm \cite{strohmer2009randomized} and the block Kaczmarz algorithm \cite{needell2014paved}. Additionally, these two Kaczmarz algorithms, especially the one presented in \cite{shao2021deterministic}, demonstrate some new and interesting geometric properties that are not obvious from the standard Kaczmarz algorithm viewpoint. These properties connect solutions of linear systems to centers of certain spheres. For example, a main result in \cite{shao2021deterministic} states that the iterative process (\ref{intro:eq1}) produces a series of points on $2m$ high-dimensional spheres, which can essentially be decomposed into $m$ pairs. For each pair, the midpoint of the centers of the spheres is exactly the solution of $A\x=\b$ when it is consistent. Thus, the problem of solving a consistent linear system of equations is closely relevant to the problem of finding the center of a sphere from knowing many points on the sphere, and the latter problem can be solved various ways. 

The two papers \cite{steinerberger2021surrounding,shao2021deterministic} provided a solid starting point for these algorithms; however, many questions remain unanswered.
For example, the analysis of the algorithm given in \cite{steinerberger2021surrounding} is mainly for invertible linear systems. However, inconsistent linear systems are more common in practice. In \cite{shao2021deterministic}, the author thoroughly studied the iterative procedure (\ref{intro:eq1}) for both consistent and inconsistent linear systems. The big difference between \cite{shao2021deterministic} and \cite{steinerberger2021surrounding} is the convergence rate. The convergence rate of the algorithm studied in \cite{shao2021deterministic} depends on a new quantity $\eta(A)$ that is closely related to the condition number $\kappa(A)$ of $A$, while the  convergence rate of \cite{steinerberger2021surrounding} (including many previous ones) depends on $\kappa_\F(A):=\|A\|_\F \|A^{-1}\|$. 
As $\kappa(A) \leq \kappa_\F(A)$, the algorithm given in \cite{shao2021deterministic} could converges faster. However, in \cite{shao2021deterministic}, few properties of $\eta(A)$ were found.

In this paper, we comprehensively study these two types of new Kaczmarz algorithms. The contributions of this paper are threefold: First, we theoretically analyse the algorithms given in \cite{steinerberger2021surrounding} for solving least squares, rather than just invertible linear systems. Second, we extend the algorithms from \cite{steinerberger2021surrounding,shao2021deterministic} to block versions and provide their theoretical convergence rates. Third, as a theoretical complement of our previous work, we address some key questions mentioned in \cite{shao2021deterministic}, particularly regarding the properties of $\eta(A)$.

We describe these results in more detail  below. To avoid confusion with the standard Kaczmarz method and some previous ones (where the constant 2 is replaced by 1 in (\ref{intro:eq1}), making each step an orthogonal projection), we refer to the iterative process in (\ref{intro:eq1}) as the reflective Kaczmarz algorithm.

\subsection{Randomised reflective Kaczmarz algorithm for least squares}

In \cite{steinerberger2021surrounding}, the author proved that if $A$ is invertible and $\x_*=A^{-1}\b$, then
\be
\label{intro:eq2}
\E\left[\left\|\x_* - \frac{1}{N}\sum_{k=0}^{N-1} \x_k \right\|^2\right] \leq \frac{1+\|A\|_\F^2 \|A^{-1}\|^2}{N} \|\x_*-\x_0\|^2.
\ee
In this paper, we show a similar result for general $A$. Now let $\x_*=A^+\b$ be a solution of the least square problem $\arg\min\|A\x-\b\|$ and assume that ${\rm Rank}(A)=n$, then (See Theorem \ref{thm:Steinerberger})
\be
\label{intro:eq3}
\E\left[\left\|\x_* - \frac{1}{N}\sum_{k=0}^{N-1} \x_k \right\|^2\right]
\leq \frac{1+\|A\|_\F^2 \|A^{-1}\|^2}{N} \|\x_*-\x_0\|^2 + O(\|\c\|^2/\sigma_n^2),
\ee
where $\c=\b-A\x_*$ is orthogonal to the column space of $A$ and $\sigma_n$ is the minimal singular value.
The convergence rates (\ref{intro:eq2}) and \eqref{intro:eq3} imply that, to ensure a good approximation, $N$ should be as large as $\|A\|_\F^2 \|A^{-1}\|^2$. This coincides with many previous Kaczmarz algorithms, such as \cite{strohmer2009randomized}. The result (\ref{intro:eq3}) is similar to the one obtained in \cite[Theorem 3.7]{zouzias2013randomized} for standard randomised Kaczmarz algorithm for least squares.

\subsection{Block reflective Kaczmarz algorithm}

In block Kaczmarz algorithm, a subset of rows of $A$ is used at each step of iteration. Block Kaczmarz algorithms can converge faster than the standard ones, as has been verified in many previous works \cite{elfving1980block, moorman2021randomized,needell2014paved,necoara2019faster,needell2015randomized}. To establish a block version of (\ref{intro:eq1}), we consider block Householder matrix, defined as follows: Let $M$ be a $r\times n$ matrix with full row rank, the block Householder matrix in terms of $M$ is defined by $\H_M := I_n - 2 M^{\TT} (MM^{\TT})^{-1} M .$
Let $Z \subseteq [m]$, we set $A_Z=\sum_{i\in Z} \e_i A_i^\T \in \mathbb{R}^{|Z|\times n}$ as the submatrix of $A$ whose rows are indexed by $Z$. The corresponding block Householder matrix is $\H_{A_Z}$. The iterative procedure now becomes
\be
\label{intro:eq5}
\x_{k+1} 
= \H_{A_Z} \x_k + 2 A_Z^+  \b_Z
= \x_k - 2A_Z^+ (A_Z \x_k 
- \b_Z), \quad 
\b_Z = \sum_{i\in Z} b_i \e_i \in \mathbb{R}^{|Z|}.
\ee
In the above, the second formula is a natural generalisation of (\ref{intro:eq1}). Regarding the first one, our goal is to have the iterative procedure satisfy the following condition: $\x_{k+1} - \x_* = \H_{A_Z} (\x_k - \x_*)$, where $\x_*$ is a solution of $A\x=\b$. Simplifying this leads to the first formula.

Regarding (\ref{intro:eq5}), there are two types of methods for choosing $Z$, each leading to a different type of block Kaczmarz algorithm.

\begin{itemize}
\item {\bf Algorithm 1:} Let $\{1,2,\ldots,m\}=Z_1 \cup \cdots \cup Z_p$ be a fixed partition of the row indices. Arbitrarily choose $\x_0$, then update
\be\label{intro:eq6}
\x_{k+1} 
= \x_k - 2 A_{Z_{i}}^+ A_{Z_{i}} \x_k + 2 A_{Z_{i}}^+  \b_{Z_{i}}, 
\quad k=0,1,2,\ldots,
\ee
where $i\in [p]$ is chosen with probability 
$\| A_{Z_{i}}\|_\F^2/\|A\|_\F^2$.

\item {\bf Algorithm 2:} Fix an integer $q\leq \min(m,n)$. Arbitrarily choose $\x_0$, then update
\be\label{intro:eq7}
\x_{k+1} 
= \x_k - 2 A_Z^+ A_Z \x_k + 2 A_Z^+  \b_Z,
\quad k=0,1,2,\ldots,
\ee
where $Z=\{i_1,\ldots,i_q\} \subseteq [m]$ is obtained in a way such that $i_j$ is collected with probability $\| A_{i_j}\|^2/\|A\|_\F^2$.
\end{itemize}

In Section \ref{section:The block version}, we analyse the convergence rates of the above two algorithms for both consistent and inconsistent linear systems, see Theorems \ref{thm:alg1-consistent-case}, \ref{thm:alg2-consistent-case}, \ref{thm:alg1-inconsistent-case} and \ref{thm:alg2-inconsistent-case} for more details. The results indicate that these two algorithms have similar convergence rates. We also numerically test their performance in Section \ref{section:Numerical demonstration}. The numerical results affirm that the efficiency of randomised block reflective Kaczmarz algorithm than the standard one. 

\subsection{More theoretical results on deterministic reflective Kaczmarz algorithm}

In the third part of this paper (see Section \ref{section: Deterministic reflective Kaczmarz algorithms}), our main focus in on the deterministic version of the iterative process (\ref{intro:eq1}), which was previously studied in \cite{shao2021deterministic}. Here, we explore more theoretical results about this algorithm. 

\begin{itemize}
    \item First, we extend this algorithm to a block version and analyse the corresponding convergence rate. Our main results (i.e., Theorems \ref{thm:convergence} and \ref{thm:0726}) indicate that all key results proved in \cite{shao2021deterministic} also hold in the block version.

    \item Second, we analyse a quantity, denoted as $\eta(A)$, see Definition \ref{defn:eta}, that is closely related to the convergence rate. In \cite{shao2021deterministic}, a upper bound of $\eta(A)$ is given, i.e., $\eta(A) = O(\kappa(A)^2 \log m)$. Here we provide more properties about this quantity. For example, in Theorem \ref{thm:eta vs sigma-min}, we will prove that
\be\label{intro:eq8}
\frac{|\eta^2-\sigma^{-2}_{\min}(A)|^2}{\eta^2\sigma^{-2}_{\min}(A)
\Big(1 + \eta^{2}\Big)\Big(1+\sigma^{-2}_{\min}(A)\Big)}
 \leq 4\|L\|^2,
\ee
where $L$ is the lower triangular matrix such that $L + L^\T = AA^\T - I$ and $\sigma_{\min}(A)$ is the minimal nonzero singular value of $A$. Particularly, if $A$ has restricted isometry property with restricted isometry constant $\delta$, then the above upper bound if $O(\delta^2 \log^2m)$. 
We also study the behavior of $\eta(A)$ under left or right orthogonal actions, see Proposition \ref{prop: some facts} and Theorem \ref{thm: some facts 2}. We will show that $\eta(A)$ is invariant under right orthogonal actions, but not under left orthogonal actions. 
In the latter case, we give a quantitative description of the change. The result is similar to (\ref{intro:eq8}). This property is in sharp contrast to singular values, which are invariant under both left and right orthogonal actions. We also provide an equation to compute $\eta(A)$, which is similar to the characteristic equation for singular values, see Theorem \ref{thm for eta}.
Interestingly, this equation was indeed previously found by Coxeter \cite{coxeter1939osung} in 1939.
\item Third, we provide further results on connections to solving least square problems. As analysed in \cite{shao2021deterministic}, the iterative process (\ref{intro:eq1}) indeed solves $A^\T L^{-1} A \x = A^\T L^{-1} \b$ rather than $A^\T A \x = A^\T\b$. Some reasons were provided. In this paper, we give more results on this and present a method to fix this, see Corollary \ref{cor:connection} and Proposition \ref{prop:connection}. 
\end{itemize}

\section{Preliminaries}

\subsection{Notation}

For a matrix $A=(a_{ij})$, its Frobenius norm is defined as $\|A\|_\F = (\sum_{i,j} |a_{ij}|^2)^{1/2}$, its operator norm is the maximum singular value, denoted as $\|A\|$ or $\sigma_{\max}(A)$. We use $\sigma_{\min}(A)$ to denote its minimum nonzero singular value, and $A^+$ to denote its pseudoinverse.
The condition number is $\kappa = \|A\| \|A^+\|$, and the scaled condition number is $\kappa_\F = \|A\|_\F \|A^+\|$. 
For a matrix $A$, we use $A^\dag, A^\T$ to denote its complex conjugate and transpose respectively.
For any two square matrices $A,B$ of the same dimension, with $A\preceq B$, we mean for any $\v$ we have $\v^\T A \v \leq \v^\T B \v$. 
We use $\{\e_1,\ldots,\e_n\}$ to represent the standard basis of $\mathbb{R}^n$. 
For any two vectors $\a,\b$, we use $\langle \a,\b\rangle$ to represent their inner product.
Let $\v$ be a vector; sometimes we will use Dirac notation $\ket{\v}$ to represent the unit vector $\v/\|\v\|$ when it makes the notation more concise and neat. Let $m$ be an integer, we use $[m]$ to denote the set $\{1,2,\ldots,m\}$.

\subsection{Perturbation of generalised eigenvalues}

Let $A, B \in \mathbb{C}^{m\times m}$ be two matrices, in the {\em generalised eigenvalue problem} \cite{stewart1990matrix}, we aim to find all $(\alpha, \beta)\neq (0,0)$ with $\alpha,\beta\in \mathbb{C}$ and $\x\in \mathbb{C}^m$ such that
\[
\beta A \x = \alpha B \x.
\]
The matrix pair $(A,B)$ is called  {\em regular} if $\det(\beta A - \alpha B)$ is not identically zero.
The distance between the eigenvalues $( \alpha_1,\beta_1 )$ and $(\alpha_2,\beta_2 )$ is measured by the {\em chordal metric}
\[
\rho(( \alpha_1,\beta_1 ), ( \alpha_2,\beta_2 ) )
= \frac{|\alpha_1\beta_2 - \alpha_2\beta_1|}{\sqrt{|\alpha_1|^2+|\beta_1|^2} \sqrt{|\alpha_2|^2+|\beta_2|^2}}.
\]

\begin{prop}[see \cite{stewart2004elsner}]
\label{prop:gep1}
Let $(A, B)$ and $(\widetilde{A}, \widetilde{B}) = (A + E, B + F )$ be regular matrix pairs, and let $( \alpha,\beta )$ be an eigenvalue of $(A, B)$, then there is an eigenvalue $( \tilde{\alpha} , \tilde{\beta} )$ of $(\widetilde{A}, \widetilde{B})$ satisfying
\[
\rho(( \alpha,\beta), ( \tilde{\alpha} , \tilde{\beta} ) )
= \frac{\sqrt{\|A\|^2 + \|B\|^2}}{\gamma(A,B)} \left( \frac{\sqrt{\|E\|^2 + \|F\|^2}}{\sqrt{\|A\|^2 + \|B\|^2}} \right)^{1/m},
\]
where
\[
\gamma(A,B) = \max_{|\alpha|^2+|\beta|^2=1} \sigma_{\min} (\beta A- \alpha B).
\]
\end{prop}

\begin{rem} \label{remark: on gamma}
{\rm 
In the above,  if $\beta=0, \alpha =-1$, then $\sigma_{\min} (\beta A + \alpha B) = \sigma_{\min}(B)$. So $\gamma(A,B) \geq  \sigma_{\min}(B)$. Similarly, $\gamma(A,B) \geq  \sigma_{\min}(A)$.
}   
\end{rem}

Let $Z=(A,B), W=(C,D)$ be regular matrix pairs with eigenvalues $(\alpha_i, \beta_i)$ and $(\gamma_i,\delta_i)$, respectively. Define
the {\em generalized spectral variation} of $W$ with respect to $Z$ by
\[
S_Z(W) := \max_i \min_j \rho((\alpha_j, \beta_j), (\gamma_i,\delta_i)).
\]
View $Z=(A,B)$ as a matrix, set $P_Z = Z^{+} Z$. We define a metric as follows:
\[
d_2(Z,W) := \|P_Z - P_W\|.
\]
By \cite[Equation (1.13)]{elsner1982perturbation}, 
\be
\label{upper bound of d2}
d_2(Z,W) \leq 
\|(ZZ^\dag)^{-1/2}\| \|Z-W\|.
\ee

\begin{prop}[Theorem 2.1 in \cite{elsner1982perturbation}]
\label{prop:gep2}
Let $Z = (A, B)$ be a diagonalisable pair, i.e., there exists matrices $S, T$ such that both $SAT$ and $SBT$
are diagonals. Let $W$ be a regular pair, then
\[
S_Z(W) \leq \|T\| \|T^{-1}\| d_2(Z,W).
\]
\end{prop}

\section{Randomised reflective Kaczmarz algorithm for least squares}

Let $A\in \mathbb{R}^{m\times n}$ and $\b\in \mathbb{R}^{m}$, for least squares, we aim to find $\tilde{\x}_*$ such that $\|\tilde{\x}_*-\x_*\|\leq \varepsilon \|\x^*\|$, where $\x_*:=\arg\min\|A\x-\b\|$.
In this section, we study the following randomised iterative scheme, which we call the {\em reflective Kaczmarz algorithm}, for solving least squares: Arbitrarily choose $\x_0$, then update via
\be
\label{standard verison}
\x_{k+1} = \x_k + 2 \frac{b_{i_k} - A_{i_k}^{\T} \x_k}{\|A_{i_k}\|^2} A_{i_k}, \quad k=0,1,2,\ldots,
\ee
where $i_k\in[m]$ is chosen with probability $\|A_{i_k}\|^2/\|A\|_\F^2$. Our main goal of this section is to prove the convergence rate of the above iterative for least squares, see Theorem \ref{thm:Steinerberger}.

For convenience, we set
\bes
\R_i = I_n - 2 \frac{A_i^\T A_i}{\|A_i\|^2}
\ees
as the reflection generated by the $i$-th row of $A$, then
\bes
\x_{k+1} = \R_{i_k} \x_k + 2 \frac{b_{i_k}}{\|A_{i_k}\|^2} A_{i_k}.
\ees

Note that when $A$ is invertible, we have $\x_{k+1}-\x_* =  \R_{i_k} (\x_k-\x_*)$, and so the vectors $\x_0,\x_1,\ldots$ lie on the sphere centered at the solution $\x_*$ with radius $\|\x_0-\x_*\|$. One simple way to approximate the centre is using the average of $\{\x_0,\x_1,\ldots,\x_{N-1}\}$ for some $N$. In \cite{steinerberger2021surrounding}, Steinerberger computed that it suffices to choose $N = O(\|A\|_\F^2 \|A^{-1}\|^2\log(1/\varepsilon))$ when $A$ is invertible. Below, we study the theoretical results of this algorithm for solving least squares.

We decompose $\b = A \x_* + \c$, where $\x_* =A^+ \b$ is an optimal solution of $\arg\min \|A\x-\b\|$, and $\c$ is orthogonal to the column space of $A$ (i.e., $A^\TT\c=0$). We denote the error in the $k$-th step of iteration as
\be
\label{notation of error}
\e_k = \x_k - \x_*,
\ee
then 
\be
\label{error-relation}
\e_{k+1} = \e_k + 2 \frac{ A_{i_k}^{\T} \x_* + c_{i_k} - A_{i_k}^{\T} \x_k}{\|A_{i_k}\|^2} A_{i_k}
= \R_{i_k} \e_k + \frac{2c_{i_k}}{\|A_{i_k}\|^2} A_{i_k}.
\ee
It is easy to check that 
\be
\label{a trivial fact}
\E_{i}\left[\frac{c_{i}}{\|A_{i}\|^2} A_{i} \right]= \sum_{i=1}^m \frac{c_i}{\|A\|_\F^2} A_i = \frac{\c^\T A}{\|A\|_\F^2} = 0.
\ee
Here we used the fact that $A^\T\c=0$.

\begin{lem}
\label{lem1}
The expectation value of $\|\e_k\|^2$ satisfies:
\be
\E\left[\|\e_k\|^2\right] = \|\e_0\|^2 + \frac{4k\|\c\|^2}{\|A\|_\F^2} .
\ee
\end{lem}

\begin{proof}
We first compute the expectation value with respect to $i_k$ using (\ref{error-relation}).
\beas
\E_{i_k}\left[\|\e_{k+1}\|^2\right] &=& 
\E_{i_k} \left[\|\R_{i_k} \e_k\|^2\right] + \E_{i_k}\left[\frac{4c_{i_k}^2}{\|A_{i_k}\|^2}\right] 
+ 
\E_{i_k} \left[ \frac{4c_{i_k}}{\|A_{i_k}\|^2} \langle A_{i_k}, \R_{i_k} \e_k\rangle \right] \\
&=& 
\|\e_k\|^2 + \frac{4\|\c\|^2}{\|A\|_\F^2} .
\eeas
Here in the last step we used (\ref{a trivial fact}) and the fact that $\langle A_{i_k}, \R_{i_k} \e_k\rangle=-\langle A_{i_k}, \e_k\rangle$.
Inductively, we obtain the claimed result.
\end{proof}

\begin{lem}
\label{lem2}
Assume that $A$ has full column rank, then for any $\ell\geq 0$
\[
\Big| \E [\langle \e_k,\e_{k+\ell}\rangle] \Big| \leq \Big(1 - 2\kappa_\F^{-2}\Big)^\ell 
\Big( \|\e_0\|^2 + \frac{4k\|\c\|^2}{\|A\|_\F^2} \Big),
\]
where $\kappa_\F = \|A\|_\F \|A^+\|$.
\end{lem}
\begin{proof} We first prove the case when $\ell=1$. By (\ref{error-relation}) and (\ref{a trivial fact}),
\beas
\E_{i_k} [\langle \e_k,\e_{k+1}\rangle] &=& 
\E_{i_k} [\langle \e_k,\R_{i_k} \e_k\rangle] + \langle \e_k, \E_{i_k} \left[\frac{2c_{i_k}}{\|A_{i_k}\|^2} A_{i_k}\right]\rangle  \\
 &=& 
\langle \e_k,\sum_{i=1}^m \frac{\|A_i\|^2}{\|A\|_\F^2} (I-2\ket{A_i} \bra{A_i}) \e_k\rangle \\
 &=& 
\langle \e_k,(I- \frac{2A^\T A}{\|A\|_\F^2}) \e_k\rangle.
\eeas
Since $A$ has full column rank, we have\footnote{If $A$ does not have full column rank, then we can only obtain $|\E_{i_k} [\langle \e_k,\e_{k+1}\rangle]| \leq \|\e_k\|^2$ here. This is not enough for convergence.}
\[
|\E_{i_k} [\langle \e_k,\e_{k+1}\rangle]| \leq (1 - \frac{2}{\|A^+\|^2\|A\|_\F^2}) \|\e_k\|^2 = 
(1 - 2\kappa_\F^{-2}) \|\e_k\|^2 .
\]
By Lemma \ref{lem1}, this further means
\[
|\E [\langle \e_k,\e_{k+1}\rangle]| \leq (1 - 2\kappa_\F^{-2}) \E[\|\e_k\|^2]
= (1 - 2\kappa_\F^{-2}) ( \|\e_0\|^2 + \frac{4k\|\c\|^2}{\|A\|_{\F}^2}).
\]
More generally, for any $\ell\geq 1$
\beas
\E [\langle \e_k,\e_{k+\ell}\rangle] &=& 
\E \Big[\E_{i_k}[\langle \e_k,\R_{i_k}\e_{k+\ell-1}\rangle] \Big] \\
&=& 
\E [\langle \e_k,\sum_{i=1}^m \frac{\|A_i\|^2}{\|A\|_\F^2}(I - 2 \ket{A_i} \bra{A_i} )\e_{k+\ell-1}\rangle] \\
&=& 
\E [\langle \e_k,(I- \frac{2A^\T A}{\|A\|_\F^2}) \e_{k+\ell-1}\rangle].
\eeas
By induction,
\beas
\E [\langle \e_k,\e_{k+\ell}\rangle]
=\E [\langle \e_k,(I- \frac{2A^\T A}{\|A\|_\F^2})^\ell \e_{k}\rangle].
\eeas
Thus,
\[
|\E [\langle \e_k,\e_{k+\ell}\rangle]| \leq (1 - 2\kappa_\F^{-2})^\ell ( \|\e_0\|^2 + \frac{4k\|\c\|^2}{\|A\|_\F^2}).
\]
This completes the proof.
\end{proof}

\begin{thm}[Convergence rate of reflective Kaczmarz algorithm for least squares]
\label{thm:Steinerberger}
Assume that $A\in \mathbb{R}^{m\times n}$ has full column rank and $\sigma_n$ is the minimal nonzero singular value.
Let $\b=A\x_*+\c$ be an orthogonal decomposition, where $\x_*=A^+\b$. Let $\{\x_0,\x_1,\x_2,\ldots\}$ be the points generated by the process \eqref{standard verison}.
Then
\bea
\E\left[\, \left\|\x_* - \frac{1}{N}\sum_{k=0}^{N-1} \x_k \right\|^2 \, \right]
&\leq& \frac{1+\kappa_\F^2}{N} \|\x_*-\x_0\|^2 + \left( 2 + 2\kappa_\F^2 +  10 \frac{\kappa_\F^6}{N^2}   \right) \frac{\|\c\|^2}{\|A\|_\F^2}
\nonumber \\
&=& \frac{1+\kappa_\F^2}{N} \|\x_*-\x_0\|^2 + 
O\left( \frac{\|\c\|^2}{\sigma_n^2} \right).
\eea
\end{thm}

\begin{proof}
Using the notation (\ref{notation of error}), it suffices to estimate
\beas
\E\left[\,\left\|\sum_{k=0}^{N-1} \e_k \right\|^2 \, \right]
&=& \E \left[ \sum_{k=0}^{N-1} \|\e_k\|^2 + \sum_{k\neq \ell} \langle \e_k, \e_\ell\rangle \right] \\
&=& N \|\e_0\|^2 + \frac{2N(N-1)\|\c\|^2}{\|A\|_\F^2}
+ 2 \sum_{k=0}^{N-2} \sum_{\ell=k+1}^{N-1} \E [\langle \e_k,\e_\ell\rangle] \\
&\leq& N \|\e_0\|^2 + \frac{2N(N-1)\|\c\|^2}{\|A\|_{\F}^2}
+ 2 \sum_{k=0}^{N-2} \sum_{\ell=k+1}^{N-1} (1 - 2\kappa_\F^{-2})^{\ell-k} \left( \|\e_0\|^2 + \frac{4k\|\c\|^2}{\|A\|_\F^2} \right).
\eeas
The second step is based on Lemma \ref{lem1} and the third step is caused by Lemma \ref{lem2}.
It is easy to compute that for any $0\leq q\leq 1$,
\beas
\sum_{k=0}^{N-2} \sum_{l=k+1}^{N-1} q^{l-k}
&=& \frac{q^{N+1} - q + Nq(1-q)}{(1-q)^2} 
~~\leq~~ \frac{N}{1-q}, \\
\sum_{k=0}^{N-2} \sum_{l=k+1}^{N-1} k q^{l-k} 
&=& \frac{q N^2 (1-q)^2+q(3-q)+2q(1-q^N)}{2(1-q)^3}
~~\leq~~ \frac{N^2}{2(1-q)} + \frac{5}{2(1-q)^3}.
\eeas
Thus, replacing $q=1-2\kappa_\F^{-2}$ gives
\beas
\E\left[\,\left\|\sum_{k=0}^{N-1} \e_k \right\|^2\,\right]
\leq N (1+\kappa_\F^2) \|\e_0\|^2 +  ( 2N(N-1) + 2N^2\kappa_\F^2 +  10\kappa_\F^6   ) \frac{\|\c\|^2}{\|A\|_\F^2}.
\eeas
Dividing both sides by $N$ yields the claimed bound.
\end{proof}

The result above coincides with the main theorem of \cite{steinerberger2021surrounding} for invertible linear systems.
The error bound obtained in Theorem \ref{thm:Steinerberger} is also very close to that of the standard Kaczmarz method for inconsistent linear systems, e.g., see \cite[Theorem 3.7]{zouzias2013randomized}, \cite[Theorem 2.1]{needell2010randomized}. All are influenced by a convergence horizon $\|\c\|^2/\sigma_n^2$. From Theorem \ref{thm:Steinerberger}, to ensure that the error is bounded by $\varepsilon^2$, we need to choose $N=O(\kappa_\F^2/\varepsilon^2)$. As analyzed in \cite{steinerberger2021surrounding,shao2021deterministic}, the dependence on $\varepsilon$ can be reduced to $\log(1/\varepsilon)$ using an idea similar to the binary method. Essentially, we run the algorithm for a while until the error is smaller than half of $\|\x_*-\x_0\|^2$, compute the average of $\x_0,\ldots,\x_{N-1}$, and then start the algorithm again with this average as a new initial point.

\section{Block reflective Kaczmarz algorithms}
\label{section:The block version}

Block Kaczmarz methods have been widely studied in the past, e.g., see \cite{elfving1980block, moorman2021randomized,needell2014paved,necoara2019faster,needell2015randomized}. Their performance is generally better than that of the standard Kaczmarz method. We below extend the algorithm (\ref{standard verison}) to block versions. 
Below, we will theoretically prove that block reflective Kaczmarz algorithms can be more efficient than the standard
reflective Kaczmarz algorithm. Additionally, in Section \ref{section:Numerical demonstration}, the numerical tests will also verify this.

First, we recall the concept of block Householder matrix.

\begin{defn}[Block Householder matrix]
Let $M$ be a $r\times n$ matrix with full row rank, the block Householder matrix in terms of $M$ is defined by
\be
\label{def:block Householder matrix}
\H_M := I_n - 2 M^{\TT} (MM^{\TT})^{-1} M = I_n - 2 M^{\TT} (M^{\TT})^+
= I_n - 2 M^+ M.
\ee
So $\H_M$ is symmetric and $M \H_M = -M, \H_M M^+ = -M^+$.
\end{defn}

Let $Z \subseteq [m]$, we set $A_Z=\sum_{i\in Z} \e_i A_i^\T \in \mathbb{R}^{|Z|\times n}$ as the submatrix of $A$ whose rows are indexed by $Z$. The corresponding block Householder matrix is $\H_{A_Z}$ or simply $\H_Z$ when it does not cause confusion. The iterative procedure in block form we are interested in is
\be
\label{block iterative procedure}
\x_{k+1} 
= \H_Z \x_k + 2 A_Z^+  \b_Z
= \x_k - 2A_Z^+ (A_Z \x_k 
- \b_Z), \quad 
\b_Z = \sum_{i\in Z} b_i \e_i \in \mathbb{R}^{|Z|}.
\ee

The two algorithms, depending on the choice of $Z$, that we will consider have been described in the introduction. For convenience, we restate them below.

\begin{itemize}
\item {\bf Algorithm 1:} Let $[m]=Z_1 \cup \cdots \cup Z_p$ be a fixed partition of the row indices. Arbitrarily choose $\x_0$, then update
\be
\label{iterative formula1}
\x_{k+1} 
= \x_k - 2 A_{Z_{i}}^+ A_{Z_{i}} \x_k + 2 A_{Z_{i}}^+  \b_{Z_{i}}, 
\quad k=0,1,2,\ldots,
\ee
where $i\in [p]$ is chosen with probability 
$\| A_{Z_{i}}\|_\F^2/\|A\|_\F^2$.

\item {\bf Algorithm 2:} Fix an integer $q\leq \min(m,n)$. Arbitrarily choose $\x_0$, then update
\be
\label{iterative formula2}
\x_{k+1} 
= \x_k - 2 A_Z^+ A_Z \x_k + 2 A_Z^+  \b_Z,
\quad k=0,1,2,\ldots,
\ee
where $Z=\{i_1,\ldots,i_q\} \subseteq [m]$ is obtained in a way such that $i_j$ is collected with probability $\| A_{i_j}\|^2/\|A\|_\F^2$.
\end{itemize}

In the following, we study the convergence rate of the above two algorithms for consistent and inconsistent linear systems separately.

\subsection{For consistent linear systems}

In this subsection, we analyse the algorithm for consistent linear systems $A\x=\b$. Let $\x_*=A^+\b$ be a solution with minimum norm and $\e_k=\x_k-\x_*$ be the error at the $k$-th step of iteration. Then it is easy to check that $\e_{k+1}=\H_{Z_i} \e_k$ when $A\x=\b$ is consistent.

We first analyse {\bf Algorithm 1}. 

\begin{lem}
\label{0105lem1}
Assume that $A$ has full column rank.
Denote $\gamma = \min_{1\leq i \leq p} \frac{\| A_{Z_{i}}\|_\F}{\| A_{Z_{i}}\|}$, then the operator norm of $\E_i[\H_{Z_i}]$ satisfies
\[
\|\E_i[\H_{Z_i}]\| \leq 1 - \frac{2\gamma^2}{\kappa_\F^2} .
\]
\end{lem}

\begin{proof}
The result follows from a standard calculation:
\beas
\E[\H_{Z_i}] &=& \sum_{i=1}^p \frac{\| A_{Z_{i}}\|_\F^2}{\|A\|_\F^2} (I_n - 2 A_{Z_{i}}^+ A_{Z_{i}}) \\
&=& I_n - \frac{2}{\|A\|_\F^2}  \sum_{i=1}^p\| A_{Z_{i}}\|_\F^2 A_{Z_{i}}^\TT (A_{Z_{i}}A_{Z_{i}}^\TT)^{-1} A_{Z_{i}} \\
&\preceq& I_n - \frac{\min_i \| A_{Z_{i}}\|_\F^2 \sigma_{\min}((A_{Z_{i}}A_{Z_{i}}^\TT)^{-1})} {\|A\|_\F^2} A^\TT A \\
&\preceq& \left(1 - \frac{2}{\kappa_\F^2} \min_{1\leq i \leq p} \frac{\| A_{Z_{i}}\|_\F^2}{\| A_{Z_{i}}\|^2} \right) I_n.
\eeas
In the first inequality, $\sigma_{\min}$ refers to the minimal singular value. In the second inequality, we used the facts that 
$\sigma_{\min}((A_{Z_{i}}A_{Z_{i}}^\TT)^{-1}) = \sigma_{\max}(A_{Z_{i}})^{-2}$ and $A^\T A \succeq \sigma_{\min}(A^\T A) I_n$ since $A$ has full column rank.
\end{proof}

\begin{lem} 
\label{0105lem2}
Assume that $A$ has full column rank, then for any $\ell\geq 0$, we have
\[
\big| \E[\langle \e_k, \e_{k+\ell}\rangle] \big| = \left(1-\frac{2\gamma^2}{\kappa_\F^2}\right)^\ell \|\e_0\|^2.
\]
\end{lem}

\begin{proof}
The claim is trivial when $\ell=0$. We next consider the case that $\ell=1$.
It is easy to see from $\e_{k+1}=\H_{Z_i} \e_k$ that $\|\e_k\|=\|\e_0\|$ for all $k$.
By Lemma \ref{0105lem1}, we have
\beas
\big|\E[\langle \e_k, \e_{k+1}\rangle] \big| = \big|\E[\langle \e_k, \E_i[\H_{Z_i}]\e_{k}\rangle]\big|
\leq \|\E_i[\H_{Z_i}]\| \|\e_k\|^2
\leq (1-\frac{2\gamma^2}{\kappa_\F^2}) \|\e_0\|^2.
\eeas
Generally, when $\ell\geq 2$ we have
\beas
\E[\langle \e_k, \e_{k+\ell}\rangle] &=& \E[\langle \e_k, \E_i[\H_{Z_i}] \e_{k+\ell-1}\rangle] \\
&=& \E[\langle \e_k, \E_i[\H_{Z_i}]^{\ell} \e_{k}\rangle] \\
&\leq& \|\E_i[\H_{Z_i}]\|^\ell \|\e_k\|^2 \\
&\leq& (1-\frac{2\gamma^2}{\kappa_\F^2})^\ell \|\e_0\|^2.
\eeas
This finishes the proof.
\end{proof}

\begin{thm}[Convergence rate of {\bf Algorithm 1} for consistent linear systems]
\label{thm:alg1-consistent-case}
Assume that the linear system $A\x=\b$ is consistent and $A$ has full column rank. Let $\x_*=A^+\b$ be a solution and $\{\x_0,\x_1,\x_2,\ldots\}$ be a series of points generated by the procedure \eqref{iterative formula1}. Then
\[
\E\left[\, \left\| \x_* - \frac{1}{N}  \sum_{k=0}^{N-1} \x_k \right\|^2 \, \right]
\leq \frac{1}{N}\left( 1 + \frac{\kappa_\F^2}{2\gamma^2} \right)
 \|\x_*-\x_0\|^2.
\]
\end{thm}

\begin{proof}
Recall that $\e_k=\x_k-\x_*$. So
\[
\left\| \sum_{k=0}^{N-1} \e_k \right\|^2 =
\sum_{k=0}^{N-1} \|\e_k\|^2 + 2 \sum_{k=0}^{N-2} \sum_{\ell=k+1}^{N-1} \langle \e_k,\e_l\rangle
=
N \|\e_0\|^2 + 2 \sum_{k=0}^{N-2} \sum_{\ell=k+1}^{N-1} \langle \e_k,\e_\ell\rangle.
\]
Denote $q=1-\frac{2\gamma^2}{\kappa_\F^2}$. By Lemma \ref{0105lem2}, we have 
\[
\E\left[ \, \sum_{k=0}^{N-2} \sum_{\ell=k+1}^{N-1} \langle \e_k,\e_\ell\rangle \, \right]
\leq \sum_{k=0}^{N-2} \sum_{\ell=k+1}^{N-1} q^{\ell-k} \|\e_0\|^2
\leq \frac{N}{1-q} \|\e_0\|^2.
\]
In the above estimation, we used an inequality in the proof of Theorem \ref{thm:Steinerberger}.
Putting it all together, we obtain the claimed bound.
\end{proof}

From the above theorem, to ensure a good approximation of $\x_*$, we have $N \approx \kappa_\F^2/\gamma^2$. Generally, $\gamma\geq 1$. In the standard version without using blocks, $\gamma = 1$, so the above theorem matches Theorem \ref{thm:Steinerberger} when $\c=0$. The above theorem also explains why the block version converges faster, especially when $\gamma \gg 1$.

Next, we analyse {\bf Algorithm 2}.

\begin{lem}
\label{lem:error1}
Let $Z$ be a collection of $q$ row indices,
then for any $\u$ that lies in the column space of $A$, we have
\[
\E_Z [\langle \u | \H_Z |\u\rangle ]
\leq \left(1-\frac{2 q }{\|A\|^2 \|A^+\|^4 \|A\|_\F^2} \right) \|\u\|^2
= \left(1 - \frac{2 q}{\kappa^2 \kappa_\F^2} \right) \|\u\|^2.
\]
Consequently, when $A$ has a full column rank, we have
\[
\|\E_Z[\H_Z]\|\leq 1 - \frac{2q}{\kappa^2 \kappa_\F^2}.
\]
\end{lem}

\begin{proof} 
We use $A_Z$ to denote the submatrix of $A$ whose rows are indexed by $Z$. 
In our notation, $\H_Z$ refers to the block Householder matrix $\H_{A_Z}$.
We use $\sigma_{\max}(A_Z)$ to denote the maximal singular value of $A_Z$.
Note that
\beas
\langle \u | \H_Z |\u\rangle
= \|\u\|^2 - 2 \langle \u | A_Z^{\TT} (A_ZA_Z^{\TT})^+ A_Z |\u\rangle 
\leq \|\u\|^2 - \frac{2}{\sigma_{\max}^2(A_Z)} \|A_Z \u\|^2.
\eeas
Regarding $\sigma_{\max}(A_Z)$, notice that $A_ZA_Z^{\TT} = 
\sum_{i,j \in Z} \langle A_i, A_j\rangle  \e_i \e_j^\T$ is a submatrix of $A A^{\TT}$, so we have $\sigma_{\max}^2(A_Z) = \|A_ZA_Z^{\TT}\| \leq \|A\|^2.$
It is also easy to see that
\[
\E[\|A_Z \u\|^2]
=\langle \u | \E \left(\sum_{i\in Z} A_i^\TT A_i \right) |\u\rangle
\geq \frac{|Z|\min_{i\in Z}\|A_i\|^2}{\|A\|_\F^2} \langle \u| A^\TT A |\u \rangle
\geq \frac{|Z|\min_{i\in Z}\|A_i\|^2}{\|A^+\|^2\|A\|_\F^2} \|\u\|^2,
\]
where $|Z|=q$ is the size of $Z$.
In the last inequality, we used the fact that $\u$ lies in the column space of $A$.
Putting it together, we obtain the claimed result.
\end{proof}

In Lemma \ref{0105lem1}, we introduced a parameter $\gamma = \min_{1\leq i \leq p} \frac{\| A_{Z_{i}}\|_\F}{\| A_{Z_{i}}\|}$. It is easy to see that $\gamma^2 \geq \min_{1\leq i \leq p}\frac{|Z_i|}{\kappa(Z_i)^2}$. Here $\kappa(Z_i)$ is the condition number of $A_{Z_i}$, a submatrix of $A$ by choosing rows of $A$ indexed by $Z_i$. 
In the proof of Lemma \ref{lem:error1}, we bounded the operator norm of $A_Z$ via that of $A$, and so obtained $|Z|/\kappa(A)^2$ in the end. These two correspond to each other as we always have $\min_{1\leq i \leq p}\frac{|Z_i|}{\kappa(Z_i)^2} \geq \frac{\min_{1\leq i \leq p}|Z_i|}{\kappa(A)^2}$. But still, the bound in Lemma \ref{0105lem1} is slightly better.

With the above lemma, we can prove a similar result to Lemma \ref{0105lem2}. With these two results, we can similarly prove the following theorem, which is comparable to Theorem \ref{thm:alg1-consistent-case}.

\begin{thm}[Convergence rate of {\bf Algorithm 2} for consistent linear systems]
\label{thm:alg2-consistent-case}
Assume the linear system $A\x=\b$ is consistent and $A$ has full column rank. Let $\x_*=A^+\b$ be a solution and $\{\x_0,\x_1,\x_2,\ldots\}$ be a series of points generated by the procedure \eqref{iterative formula2}. Then
\be
\E \left[ \, \left\| \x_* - \frac{1}{N}\sum_{k=0}^{N-1} \x_k \right\|^2 \, \right]
\leq \frac{1}{N} 
\left(1  + \frac{\kappa^2 \kappa_\F^2}{2q} \right) \|\x_*-\x_0\|^2,
\ee
where $q$ is the integer specified in \eqref{iterative formula2}.
\end{thm}

\begin{proof}
The proof is similar to that of Theorem \ref{thm:alg1-consistent-case}. Note that
\beas
\E  \left[\, \left\| \sum_{k=0}^{N-1} \e_k \right\|^2\, \right]
&=& \E \left[\,\sum_{k=0}^{N-1} \|\e_k\|^2\,\right] + 2\sum_{k=0}^{N-2} \sum_{\ell=k+1}^{N-1} \E [\langle \e_k, \e_\ell\rangle] \\
&=& N \|\e_0\|^2 + 2\sum_{k=0}^{N-2} \sum_{\ell=k+1}^{N-1} \E [\langle \e_k, \e_\ell\rangle].
\eeas
By Lemma \ref{lem:error1}, for any $c\geq 1$, we have
\beas
\E [\langle \e_k|\e_{k+c} \rangle ]
&=& \E \Big[\E_Z [\langle \e_k| \P_{Z}| \e_{k+c-1} \rangle] \Big] \\
&=&
\langle \e_k| (\E_Z[\H_{Z}])^c| \e_{k} \rangle \\
&\leq& \|\E_Z[\H_{Z}]\|^c \|\e_k\|^2 \\
&\leq&
\left(1 - \frac{2q}{\kappa^2 \kappa_\F^2} \right) 
\|\e_0\|^2.
\eeas
Therefore,
\beas
\E \left[\,\left\| \frac{1}{N}\sum_{k=0}^{N-1} \e_k \right\|^2\,\right]
&\leq& \frac{1}{N} \|\e_0\|^2 + \frac{2}{N^2} \sum_{k=0}^{N-2} \sum_{\ell=k+1}^{N-1} \left(1 - \frac{2q}{\kappa^2 \kappa_\F^2}\right)^{\ell-k} \|\e_0\|^2 \\
&\leq& \frac{1}{N} \left(1  + \frac{\kappa^2 \kappa_\F^2}{2q} \right) \|\e_0\|^2,
\eeas
which is as claimed.
\end{proof}

\subsection{For inconsistent linear systems}

In this subsection, we consider {\bf Algorithm 1} and {\bf Algorithm 2} for solving inconsistent linear systems, i.e., for solving least squares. The error analysis is similar but slightly harder than that of Theorems \ref{thm:Steinerberger}, \ref{thm:alg1-consistent-case} and \ref{thm:alg2-consistent-case}. So we will make an extra but reasonable assumption in the analysis to simplify the argument. We will explain the reason after the proof of Lemma \ref{0106lem1}. 
Our main results (i.e., Theorems \ref{thm:alg1-inconsistent-case} and \ref{thm:alg2-inconsistent-case}) coincide with 
Theorems \ref{thm:alg1-consistent-case} and \ref{thm:alg2-consistent-case} when the linear system is consistent.

Let $\x_* = A^+\b$ be an optimal solution of $\arg\min\|A\x-\b\|$, then there is a $\c$ such that $\b$ admits an orthogonal decomposition $\b= A\x_*+\c$. For any $Z\subseteq [m]$, we also have $\b_Z = A_Z\x_*+\c_Z$. Here $\b_Z, \c_Z$ are the sub-vectors of $\b,\c$ respectively restricted to indices in $Z$, see \eqref{block iterative procedure}.
Let the error at the $k$-th step of iteration be $\e_k = \x_k - \x_*$. 

As usual, we start from {\bf Algorithm 1}.

\begin{lem}
\label{0106lem1}
Assume that $\E[\|\e_k\|]\leq E$ for all $k$. Let
$R=\frac{ \|\c\| }{\|A\|_\F} \max_i \|A_{Z_i}\|_\F \|A_{Z_i}^+\|$,
then 
\bea
\E [\|\e_k\|^2] &\leq& \|\e_0\|^2 + 4kR(R+E), \label{0106lem1-eq1} \\
\E [\langle \e_k,\e_{k+\ell}\rangle ] &\leq& (1 - \frac{2\gamma^2}{\kappa_\F^2})^\ell \|\e_0\|^2 
+
 4kR(R+E) (1 - \frac{2\gamma^2}{\kappa_\F^2})^\ell 
+ \frac{RE \kappa_\F^2}{\gamma^2} \left(1-(1 - \frac{2\gamma^2}{\kappa_\F^2})^\ell\right),
\label{0106lem1-eq2}
\eea
where $\gamma$ is defined in Lemma \ref{0105lem1}.
\end{lem}

\begin{proof}
From (\ref{iterative formula1}), we have
\[
\e_{k+1} = \x_k - \x_* - 2 A_{Z_i}^+ A_{Z_i} \x_k 
+ 2 A_{Z_i}^+ ( A_{Z_i}\x_*+\c_{Z_i}) 
= \H_{Z_i} \e_k + 2 A_{Z_i}^+ \c_{Z_i}.
\]
Note that
\beas
\E_i\left[\|A_{Z_i}^+ \c_{Z_i}\|^2\right] 
&=& \E_i\left[\c_{Z_i}^\TT (A_{Z_i} A_{Z_i}^\TT )^{-1}\c_{Z_i}\right]  \\
&=&
\frac{\sum_i \|A_{Z_i}\|_\F^2 }{\|A\|_\F^2} 
\c_{Z_i}^\TT (A_{Z_i} A_{Z_i}^\TT )^{-1}\c_{Z_i} \\
&\leq& 
\frac{ \|\c\|^2 }{\|A\|_\F^2} \max_i \|A_{Z_i}\|_\F^2 \|A_{Z_i}^+\|^2.
\eeas
Namely, $\E_i\left[\|A_{Z_i}^+ \c_{Z_i}\|^2\right] \leq R^2$.
As $\H_{Z_i}A_{Z_i}^+=-A_{Z_i}^+$, we have 
\beas
\E [ \|\e_{k+1}\|^2 ]
&=& \E \Big[\E_i[\|\e_k - 2A_{Z_i}^+ \c_{Z_i}\|^2 ] \Big] \\
&=&  \E \Big[ \|\e_k\|^2 + 4 \E_i[\|A_{Z_i}^+ \c_{Z_i}\|^2 ]
- 4 \langle \e_k, \E_i[A_{Z_i}^+ \c_{Z_i}] \rangle \Big] \\
&\leq&
\E[\|\e_k\|^2 ] + (4R^2+4RE).
\eeas
Inductively, we obtain the first claim. 

We below prove the second claim.
Denote $\H=\E_i[\H_{Z_i}]$ for convenience. Then by Lemma \ref{0105lem1}, $\rho := \|\H\| \leq 1 - \frac{2\gamma^2}{\kappa_\F^2}.$
We can compute that
\beas
\E [\langle \e_k,\e_{k+\ell}\rangle ] &=& \E \Big[\Big\langle \e_k, \E_i[\H_{Z_i}\e_{k+\ell-1} + 2A_{Z_i}^+ \c_{Z_i} ] \Big\rangle  \Big] \\
&\leq& \E [\langle \e_k\H, \e_{k+\ell-1}\rangle] + 2RE \\
&=& \E \Big[\Big \langle \e_k\H, \E_i[\H_{Z_i}\e_{k+\ell-2} + 2A_{Z_i}^+ \c_{Z_i} ]  \Big\rangle  \Big]  + 2RE \\
&\leq& \E [\langle \e_k\H^2, \e_{k+\ell-1}\rangle] + 2RE(1+\rho).
\eeas
Continuous this process, we will end up with
\beas
\E [\langle \e_k,\e_{k+\ell}\rangle ] &\leq& \E [\langle \e_k\H^\ell, \e_{k}\rangle] + 2RE(1+\rho+\cdots+\rho^{\ell-1}) \\
&\leq& \rho^\ell \E [\|\e_k\|^2] + 2RE \frac{1-\rho^\ell}{1-\rho}.
\eeas
By the first claim, we obtain
\beas
\E [\langle \e_k,\e_{k+\ell}\rangle ] 
\leq \rho^\ell (\|\e_0\|^2 + 4kR(R+E)) + 2RE \frac{1-\rho^\ell}{1-\rho}.
\eeas
Simplifying this leads to the second claim.
\end{proof}

\begin{rem}{\rm 
(1). In the proof of the above lemma, we may not have $\E_i[A_{Z_i}^+ \c_{Z_i}] = 0$.
This is in sharp contrast to Lemma \ref{lem1} (more precisely \eqref{a trivial fact}) or the block version using orthogonal projections \cite{needell2014paved}. This makes the analysis harder. So in the above upper bound analysis of $\|\e_{k+1}\|^2$, we used a quite large bound of $\|\e_k\|^2 + 4R(R+E)$ for simplicity. But when return to the situation of non-block version, we can  remove the assumption $\E[\|\e_k\|] \leq E$ and
set $E=0$ directly in the above proof. Consequently, it coincides with Lemma \ref{lem1}. 

(2). The assumption $\E[\|\e_k\|] \leq E$ in Lemma \ref{0106lem1} may be a little strange at first glance. However, the following two reasons indicate the rationality of this assumption. First, we have the following very coarse upper bound
\[
\E[\|\e_k\|]=\E[\|\e_{k-1}-2A_{Z_i}^+ \c_{Z_i}\|]
\leq \E[\|\e_{k-1}\|] + 2R \leq \|\e_0\| + 2kR.
\]
Namely, we can replace $E$ with $E_k:=|\e_0\| +2kR$ in the above estimate, which depends on $k$ now. But in practice, the error should decrease on average. The replacement will affect Theorem \ref{thm:alg1-inconsistent-case} below by a factor of $N$ for the error term (i.e., Err), which is acceptable when $\|\c\|\approx 0$.
Second, the bounds on the error terms in Lemma \ref{0106lem1} are affected by $R$, which further depends on $\|\c\|$. When $\|\c\|$ is small, i.e., when $\b$ is close to the column space of $A$, then $R$ can be very small. As we can see from the error term below, $E$ is multiplied by $R$ too. So in this case, the overall error can still be very small even if $E$ is replaced by $\|\e_0\| +2k R$ at the $k$-th step of iteration.
}\end{rem}

\begin{thm}[Convergence rate of {\bf Algorithm 1} for inconsistent linear systems]
\label{thm:alg1-inconsistent-case}
Assume that $A$ has full column rank. 
Let $\x_*=A^+\b$ and $\{\x_0,\x_1,\x_2,\ldots,\x_{N-1}\}$ be a series of points generated by the procedure \eqref{iterative formula1}. Assume that $\E[\|\x_*-\x_k\|] \leq E$ for all $k\leq N-1$.
Let
$R=\frac{ \|\c\| }{\|A\|_\F} \max_i \|A_{Z_i}\|_\F \|A_{Z_i}^+\|$.
Then
\[
\E\left[ \, \left\| \x_* - \frac{1}{N}\sum_{k=0}^{N-1} \x_k \right\|^2 \, \right]
\leq 
\frac{1}{N} \left( 1+\frac{\kappa_\F^2}{2\gamma^2} \right) \|\x_*-\x_0\|^2
+{\rm Err},
\]
where
${\rm Err}=O(R(R+E)\kappa_\F^2/\gamma^2)$ and  $\gamma$ is defined in Lemma \ref{0105lem1}.
\end{thm}

\begin{proof}
Recall that $\e_k=\x_k-\x_*$ and 
\beas
\E\left[ \, \left\|\sum_{k=0}^{N-1} \e_k \right\|^2 \, \right]
= \sum_{k=0}^{N-1} \E[\|\e_k\|^2] + 2 \sum_{k=0}^{N-2} \sum_{\ell=1}^{N-k-1} \E [\langle \e_k,\e_{k+\ell}\rangle].
\eeas
For the first term, by Equation (\ref{0106lem1-eq1}), we have
\[
\sum_{k=0}^{N-1} \E[\|\e_k\|^2]
\leq N\|\e_0\|^2 + 2N^2R(R+E).
\]
For the second term, by Equation (\ref{0106lem1-eq2}), we have
\beas
\sum_{k=0}^{N-2} \sum_{\ell=1}^{N-k-1} \E [\langle \e_k,\e_{k+\ell}\rangle]
&\leq& 
\sum_{k=0}^{N-2} \sum_{\ell=1}^{N-k-1} (1 - \frac{2\gamma^2}{\kappa_\F^2})^\ell \|\e_0\|^2 \\
&& +\sum_{k=0}^{N-2} \sum_{\ell=1}^{N-k-1}
 4kR(R+E) (1 - \frac{2\gamma^2}{\kappa_\F^2})^\ell  \\
&& +\sum_{k=0}^{N-2} \sum_{\ell=1}^{N-k-1} \frac{RE \kappa_\F^2}{\gamma^2} \left(1-(1 - \frac{2\gamma^2}{\kappa_\F^2})^\ell\right).
\eeas
From the proof of Theorem \ref{thm:Steinerberger}, it is easy to estimate that
\[
\sum_{k=0}^{N-2} \sum_{\ell=1}^{N-k-1} (1 - \frac{2\gamma^2}{\kappa_\F^2})^\ell
\leq 
\frac{N\kappa_\F^2}{2\gamma^2}.
\]
Similarly, we have
\[
\sum_{k=0}^{N-2} \sum_{\ell=1}^{N-k-1} 4kR(R+E) (1 - \frac{2\gamma^2}{\kappa_\F^2})^\ell
\leq 2R(R+E) \frac{N^2\kappa_\F^2}{\gamma^2},
\]
and
\[
\sum_{k=0}^{N-2} \sum_{\ell=1}^{N-k-1} \frac{RE \kappa_\F^2}{\gamma^2} \left(1-(1 - \frac{2\gamma^2}{\kappa_\F^2})^\ell\right)
\leq
\frac{N^2RE \kappa_\F^2}{\gamma^2}.
\]
Finally, putting it all together, we obtain the claimed result.
\end{proof}

In the end, we consider {\bf Algorithm 2}. The proofs are similar. Now we have to use Lemma \ref{0105lem2} to compute a bound of the operator norm. Basically, we have the following result. The proof is similar to that of Lemma \ref{0106lem1}.

\begin{lem}
\label{0106lem2}
Assume that $\E[\|\e_k\|]\leq E$ for all $k$. Let $q$ be an integer specified in {\bf Algorithm 2}. Let
\be
\label{0107 bound of R}
R=\frac{ \sqrt{q} \|A\|\|\c\| }{\|A\|_\F} \max_{Z\subseteq [m]:|Z|=q}\|(A_ZA_{Z})^+\|^{1/2}.
\ee
then 
\beas
\E [\|\e_k\|^2] &\leq& \|\e_0\|^2 + 4kR(R+E),  \\
\E [\langle \e_k,\e_{k+\ell}\rangle ] &\leq& (1 - \frac{2q}{\kappa^2 \kappa_\F^2})^\ell \|\e_0\|^2 
+
 4kR(R+E) (1 - \frac{2q}{\kappa^2 \kappa_\F^2})^\ell 
+ \frac{RE \kappa_\F^2}{\gamma} \left(1-(1 - \frac{2q}{\kappa^2 \kappa_\F^2})^\ell\right).
\eeas
\end{lem}

Finally, similar to the proof of Theorem \ref{thm:alg1-inconsistent-case}, by using Lemma \ref{0106lem2} we will obtain the following result.

\begin{thm}[Convergence rate of {\bf Algorithm 2} for inconsistent linear systems]
\label{thm:alg2-inconsistent-case}
Assume that $A$ has full column rank. 
Let $\x_*=A^+\b$ and $\{\x_0,\x_1,\x_2,\ldots,\x_{N-1}\}$ be a series of points generated by the procedure \eqref{iterative formula2}. Assume that $\E[\|\x_*-\x_k\|] \leq E$ for all $k\leq N-1$.
Let
$R$ be given in \eqref{0107 bound of R} and $q$ be an integer specified in {\bf Algorithm 2}, then
\[
\E\left[ \left\|\x_*- \frac{1}{N}\sum_{k=0}^{N-1} \x_k \right\|^2 \right]
\leq 
\frac{1}{N} \left( 1+\frac{\kappa_\F^2 \kappa^2}{2q} \right) \|\x_*-\x_0\|^2
+ {\rm Err},
\]
where
${\rm Err}=O(R(R+E)\kappa_\F^2\kappa^2/q)$.
\end{thm}

\section{Deterministic reflective Kaczmarz algorithm}
\label{section: Deterministic reflective Kaczmarz algorithms}

In \cite{shao2021deterministic}, the author studied a deterministic version of the algorithm (\ref{standard verison}), where $i_k$ is chosen cyclically, i.e., $i_k=(k \mod m) + 1$. Some interesting and new geometric properties about solutions of linear systems were discovered. 
Roughly, the two main results of \cite{shao2021deterministic} are: (1). Let $\{\x_0,\x_1,\x_2,\ldots\}$ be the points generated by (\ref{standard verison}) with $i_k=(k \mod m) + 1$, then the average of $\{\x_{jk}:j=0,1,\ldots,N-1\}$ is an approximation of the solution $\x_*$ of $A\x=\b$ that is closest to $\x_0$ when it is consistent. To ensure a good approximation, $N$ is determined by a quantity, denoted as $\eta(A)$ in \cite{shao2021deterministic}, also see Definition \ref{defn:eta1}. 
In \cite{shao2021deterministic}, it was proved that $\eta(A)=O(\kappa(A)^2)$.
(2). Moreover, 
$\{\x_{2jk}:j=0,1,2,\ldots\} \cup \{\x_{(2j+1)k}:j=0,1,2,\ldots\}$ lie on two spheres such that the midpoint of the two centers of spheres is $\x_*$ exactly.

In this section, as an continuation of this research, we investigate more properties of this algorithm. First, we study the block version of the algorithm. 
We will show that all results proved in \cite{shao2021deterministic} also hold in the block version.
Second, as a theoretical complement to this algorithm, we further study more properties about the convergence rate (i.e., $\eta(A)$) and its connections to least squares.

\subsection{Block version}

Let $A\in \mathbb{R}^{m\times n}$. As discussed above (similar to {\bf Algorithm 1}), in the block version we decompose $A$ into several parts according to a row partition of $A$. More precisely,
let $Z = \{Z_1,\ldots,Z_p\}$ be a set of submatrices that form a row partition of $A$, i.e., each $Z_i$ is a submatrix of $A$ generated by some rows of $A$. So there is a permutation matrix $S\in \mathbb{R}^{m\times m}$ such that
\be
SA = \begin{pmatrix}
    Z_1 \\
    Z_2 \\
    \vdots \\
    Z_p
\end{pmatrix}.
\label{row partition of A}
\ee
The partition $Z$ will be called a {\em row partition} of $A$ for convenience. 

\begin{rem}{\rm 
In the following analysis, for convenience, we shall set $S=I_m$ and assume that the partition is contiguous. All results stated below hold for a general permutation  $S$  and when the partition is not contiguous.
}\end{rem}

In this section, we mainly consider the following deterministic block algorithm for solving linear systems:
\be
\label{deterministic block algorithm}
\x_{k+1} 
= \x_k - 2 A_{Z_{i}}^+ A_{Z_{i}} \x_k + 2 A_{Z_{i}}^+  \b_{Z_{i}}, 
\quad k=0,1,2,\ldots,
\ee
where $i=(k \mod p)+1$ is chosen cyclically.

From (\ref{row partition of A}), when $S=I_m$, it is easy to see that
\[
AA^{\TT} = \begin{pmatrix} \vspace{.1cm}
    Z_1Z_1^{\TT} & Z_1Z_2^{\TT} & \cdots & Z_1Z_p^{\TT} \\\vspace{.1cm}
    Z_2Z_1^{\TT} & Z_2Z_2^{\TT} & \cdots & Z_2Z_p^{\TT} \\\vspace{.1cm}
    \vdots & \vdots & \ddots & \vdots \\
    Z_pZ_1^{\TT} & Z_pZ_2^{\TT} & \cdots & Z_kZ_p^{\TT}
\end{pmatrix} \in \mathbb{R}^{m\times m}.
\]
A matrix $X$ is called a {\em block-partition matrix} of $A$ with respect to partition $Z$ if it has the form
\be
\label{Z-blocked matrix}
X = \begin{pmatrix}\vspace{.1cm}
    Z_1Z_1^{\TT} & 0 & \cdots & 0 \\\vspace{.1cm}
    2Z_2Z_1^{\TT} & Z_2Z_2^{\TT} & \cdots & 0 \\\vspace{.1cm}
    \vdots & \vdots & \ddots & \vdots \\
    2Z_pZ_1^{\TT} & 2Z_pZ_2^{\TT} & \cdots & Z_pZ_p^{\TT}
\end{pmatrix} \in \mathbb{R}^{m\times m}.
\ee
Basically, $X$ is a blocked lower triangular matrix satisfying $X+X^{\TT} = 2 A A^{\TT}$. It is easy to see that $\det(X) = \prod_{i=1}^p \det(Z_iZ_i^\TT)$. So $X$ is nonsingular when all $Z_i$ have full row rank.

\begin{lem}
\label{lem:a property of blocked matrix}
Let $X$ be a block-partition matrix of $A$ of the form \eqref{Z-blocked matrix}. We decompose it as
\[
X = \begin{pmatrix}\vspace{.1cm}
    X_1 & 0 \\
    Z_p X_2 & Z_pZ_p^\TT
\end{pmatrix},
\]
where $X_1$ is the top-left corner containing the first $(p-1)\times (p-1)$ blocks, then the pseudoinverse of $X$ satisfies
\[
X^+ = \begin{pmatrix}\vspace{.1cm}
    X_1^+ & 0 \\
    -(Z_pZ_p^\TT)^+ Z_p X_2 X_1^+ & (Z_pZ_p^\TT)^+
\end{pmatrix}.
\]
Moreover,
\[
XX^+ 
=\begin{pmatrix}
    X_1X_1^+ & 0 \\
    0 & (Z_pZ_p^\TT)(Z_pZ_p^\TT)^+
\end{pmatrix}, \quad
X^+X 
=\begin{pmatrix}
    X_1^+X_1 & 0 \\
   0 & (Z_pZ_p^\TT)^+(Z_pZ_p^\TT)
\end{pmatrix}.
\]
Consequently, we have  $X_2 (I-X_1^+X_1)=0$ and $(I - XX^+)A=0$.
\end{lem}

\begin{proof}
We prove this lemma by induction on $k$. The claim is obviously true if $k=1$. We below prove the general case.
It is easy to compute that
\[
XX^+ = \begin{pmatrix}\vspace{.1cm}
    X_1X_1^+ & 0 \\
    [Z_p  - (Z_kZ_k^\TT)(Z_kZ_k^\TT)^+ Z_p] X_2 X_1^+  & (Z_pZ_p^\TT)(Z_pZ_p^\TT)^+
\end{pmatrix}
=\begin{pmatrix}\vspace{.1cm}
    X_1X_1^+ & 0 \\
    0 & (Z_pZ_p^\TT)(Z_pZ_p^\TT)^+
\end{pmatrix}.
\]
In the above, we used the fact that $M  - (MM^\TT)(MM^\TT)^+ M=0$ for any matrix $M$. 

We also have
\[
X^+X = \begin{pmatrix}\vspace{.1cm}
    X_1^+X_1 & 0 \\
    (Z_pZ_p^\TT)^+ Z_p X_2 (I-X_1^+X_1) & (Z_pZ_p^\TT)^+(Z_pZ_p^\TT)
\end{pmatrix}
=\begin{pmatrix}\vspace{.1cm}
    X_1^+X_1 & 0 \\
   0 & (Z_pZ_p^\TT)^+(Z_pZ_p^\TT)
\end{pmatrix}
\]
due to $X_2 (I-X_1^+X_1)=0$. This indeed can be proved by induction at the same time. To be more exact, we denote
\[
X_1 = \begin{pmatrix}\vspace{.1cm}
    Y_1 & 0 \\
    Z_{p-1} Y_2 & Z_{p-1}Z_{p-1}^\TT
\end{pmatrix}.
\]
By induction, we have $Y_2 (I-Y_1^+Y_1)=0$ and
\[
X_1^+ = \begin{pmatrix}\vspace{.1cm}
    Y_1^+ & 0 \\
    -(Z_{p-1}Z_{p-1}^\TT)^+ Z_{p-1} Y_2 Y_1^+ & (Z_{p-1}Z_{p-1}^\TT)^+
\end{pmatrix}.
\]
Note that $X_2 = 2 ( Z_1^\TT, \cdots, Z_{p-2}^\TT, Z_{p-1}^\TT) = (Y_2, 2Z_{p-1}^\TT)$, so
\beas
X_2 (I-X_1^+X_1) &=&
(Y_2, 2Z_{p-1}^\TT)
\begin{pmatrix}\vspace{.1cm}
    I-Y_1^+Y_1 & 0 \\
   0 & I-(Z_{p-1}Z_{p-1}^\TT)^+
\end{pmatrix} \\
&=& \begin{pmatrix}\vspace{.1cm}
    Y_2(I-Y_1^+Y_1) \\
   2Z_{p-1}^\TT(I-(Z_{p-1}Z_{p-1}^\TT)^+)
\end{pmatrix} = 0.
\eeas
It is now easy to verify the properties of pseudoinverse for $X$, namely, $XX^+X =X, X^+XX^+ = X^+$, and $XX^+, X^+X$ are Hermitian.


We now denote $A=\begin{pmatrix}
    A_1 \\
    Z_k
\end{pmatrix}$, where $A_1$ corresponds to $X_1$. 
By induction we have
\[
XX^+ A = 
\begin{pmatrix}\vspace{.1cm}
    X_1X_1^+ A_1  \\
    (Z_pZ_p^\TT)(Z_pZ_p^\TT)^+ Z_p
\end{pmatrix} = 
\begin{pmatrix}\vspace{.1cm}
    A_1 \\
    Z_p
\end{pmatrix} = A.
\]
This completes the proof.
\end{proof}

\begin{defn}
Let $Z = \{Z_1,\ldots,Z_p\}$ be a  row partition of $A$. Define
\be
\H_{Z}(A) := \H_{Z_p} \cdots \H_{Z_2} \H_{Z_1}
\ee
as the product of block Householder matrices. When it makes no confusion, we sometimes simply denote it as $\H(A)$. Note that $\H_A$ and $\H(A)$ are different.
\end{defn}

\begin{lem}
\label{lem:generalization of BW formula}
Let $Z = \{Z_1,\ldots,Z_p\}$ be a contiguous row partition of $A$. 
Then $\H_{Z}(A) = I_n - 2 A^{\TT} X^+ A$, where $X$ is the block-partition matrix of $A$ with respect to $Z$ defined by \eqref{Z-blocked matrix}.
\end{lem}

\begin{proof}
We prove the result by induction on $p$. For generality, we include $S$ in the proof. If $p=1$, then $S=I, Z_1=A, X=Z_1Z_1^\TT$, the claim is obvious. Next, let $Z_A, Z_B$ be partitions of $A,B$ respectively, then by induction $\H_{Z}(A)=I_n - 2 (S_1A)^{\TT} X_A^+ (S_1A)$ and $\H_{Z}(B)=I_n - 2 (S_2B)^{\TT} X_B^+ (S_2B)$ for some permutations $S_1, S_2$ and two block-partition matrices $X_A, X_B$. Denote
$C=\begin{pmatrix}
    A \\
    B
\end{pmatrix}, 
S = \begin{pmatrix}
    S_1 & \\
    & S_2
\end{pmatrix}$, then
\[
2(SC)(SC)^\TT
=\begin{pmatrix}\vspace{.1cm}
    2(S_1A) (S_1A)^\TT & 2(S_1A) (S_2B)^\TT \\
    2(S_2B) (S_1A)^\TT & 2(S_2B) (S_2B)^\TT
\end{pmatrix}
=\begin{pmatrix}\vspace{.1cm}
    X_A+X_A^{\TT} & 2(S_1A) (S_2B)^\TT \\
    2(S_2B) (S_1A)^\TT & X_B+X_B^{\TT}
\end{pmatrix}.
\]
By Lemma \ref{lem:a property of blocked matrix}, the block-partition matrix, which is denoted as $X_C$, of $C$ satisfies
\[
X_C = \begin{pmatrix}\vspace{.1cm}
    X_A  & 0 \\
    2(S_2B) (S_1A)^\TT & X_B 
\end{pmatrix},
\quad 
X_C^+ = \begin{pmatrix}\vspace{.1cm}
    X_A^+  & 0 \\
    -2X_B^+(S_2B) (S_1A)^\TT X_A^+ & X_B^+
\end{pmatrix}.
\]
Now it is easy to check that $\H_{Z}(C) = \H_{Z}(B) \H_{Z}(A) = I_n - 2 (SC)^\TT X_C^+ (SC)$.
\end{proof}

\begin{rem}{\rm
When each $Z_i$ contains only one row of $A$, 
the decomposition of $\H_{Z}(A)$ is known as the $WY$ representation for the product of Householder transformations
\cite{schreiber1989storage}. In our notation, $W = A^\TT X^+, Y = A$. The decomposition we obtained corresponds to the $Y^\TT T Y$ decomposition mentioned in \cite{schreiber1989storage}, i.e., $T = X^+$ in our notation.
}\end{rem}

\begin{defn}
\label{defn:eta}
Fix a partition $Z=\{Z_1,\ldots,Z_p\}$ of the rows of $A$. Assume that the eigenvalues of $\H_{Z}(A)$ are $e^{2i\theta_l}$ with $\theta_l\in [0,\pi)$ for $ l\in\{1,2,\ldots,n\}$. Define
\be
\eta_Z(A) := \frac{1}{\min_{\ell:\theta_\ell\neq 0} |\sin\theta_\ell|}.
\ee
\end{defn}

Some properties of $\eta(A)$ when each $Z_i$ contains one row of $A$ has been given in \cite{shao2021deterministic}.  In the next subsection, we shall study more properties about $\eta_Z(A)$. The proof of the following theorem is the same as that of \cite[Theorem 3.4]{shao2021deterministic}. 

\begin{thm}
\label{thm:convergence}
Assume that the linear system $A\x=\b$ is consistent.
Fix a partition $Z=\{Z_1,\ldots,Z_p\}$ of the rows of $A$.
Let $\x_0,\x_1,\ldots,\x_{N-1}$ be the vectors generated by the process \eqref{deterministic block algorithm}.
Then
\be
\left\| \x_* - \frac{1}{N} \sum_{j=0}^{N-1} \x_{jk} \right\|
\leq \varepsilon \|\x_* - \x_0\|,
\ee
where $\x_*$ is the solution that has the minimal distance to $\x_0$ and
\be
N = \left\lceil \frac{\pi \eta_Z(A)}{\varepsilon} \right\rceil.
\ee
\end{thm}

Suppose $Z=\{Z_1,Z_2\}$ and $Z_1$ has the full column rank so that $\H_{Z_1}=-I$, then the set of $S_1=\{\x_{2j}:j=0,1,\ldots,\} = \{\x_0, \x_* - \H_{Z_2}(\x_0-\x_*)\}$. The other set is $S_2=\{\x_{2j+1}:j=0,1,\ldots,\} = \{2\x_*-\x_0, \x_* + \H_{Z_2}(\x_0-\x_*)\}$. Now $S_1$ or $S_2$ alone is not enough for us to compute $\x_*$. However, their union is enough.

As inspired by our previous research \cite{shao2021deterministic}, we hope the algorithm (\ref{deterministic block algorithm}) achieves the following: if we start from $\x_0$, then the algorithm will always converge to the solution $\x_*$ that is closest to $\x_0$. As shown in \cite{shao2021deterministic}, a necessary and sufficient condition for this is that $\x_0-\x_*$ has no overlap on the eigenspace of $\H_{Z}(A)$ corresponding to the eigenvalue 1. We will prove this in the block version and provide a more rigorous analysis.

\begin{prop}
Assume that $A\x=\b$ is consistent. 
Let $\{\x_0,\x_1,\ldots\}$ be the vectors generated by the process \eqref{deterministic block algorithm}, and let $\x_*$ be the solution that has the minimal distance to $\x_0$. 
Then the following statements are equivalent.
\begin{enumerate}
    \item $\lim_{N\rightarrow \infty} \frac{1}{N} \sum_{k=0}^{N-1}\x_k = \x_*$.
    \item $\x_0-\x_*$ has no overlap on the eigenspace of $\H_{Z}(A)$ of eigenvalue 1.
    \item For any $\v$, 
    \be
    \label{condition}
    \H_{Z}(A)\v=\v \Leftrightarrow A\v=0.
    \ee
\end{enumerate}
\end{prop}
\begin{proof}
$(1 \Leftrightarrow 2)$ 
Suppose $\y_0:=\x_0-\x_*=\u + \v$, where $\u$ lies in the eigenspace of $\H_{Z}(A)$ of eigenvalues not equal 1, and $\v$ lies in the eigenspace of $\P_{Z,A}$ of eigenvalue 1. In the iteration, we obtain $\y_k:=\x_k-\x_* = \H_{Z}(A)^k \u + \v$. As a result,
\[
\left\|\x_* - \frac{1}{N} \sum_{k=0}^{N-1} \x_k\right\|^2  = \frac{1}{N^2}\left\| \sum_{k=0}^{N-1} \H_{Z}(A)^k \u \right\|^2 + \|\v\|^2 \leq \frac{\pi^2\eta_Z^2(A)}{N^2}\|\u\|^2 + \|\v\|^2,
\]
where the first estimate is similar to the proof of Theorem \ref{thm:convergence}. Thus $\v=0$. 

$(3 \Rightarrow 2)$ Consider the same decomposition as above $\x_0-\x_*=\u + \v$. Note that $\|\x_0 - (\x_*+\v)\|^2 = \|\u\|^2 \leq \|\u\|^2 + \|\v\|^2 = \|\x_0-\x_*\|$. If $\H_{Z}(A)\v=\v \Leftrightarrow A\v=0$, then $\x_*+\v$ is also a solution of $A\x=\b$. By the property of $\x_*$ we have $\v=0$. 

$(2 \Rightarrow 3)$ Suppose $\H_{Z}(A)\v=\v$. We apply the algorithm to $\v$, we will obtain a solution $\x_1$ such that $\v-\x_1=\u_1+\v_1$. By statement 2, $\v_1=0$. Similarly, if we start from $2\v$, then we will obtain $\x_2$ with $2\v-\x_2=\u_2$. So $A\v + A(\u_1-\u_2)=0$. Since $A\w=0 \Rightarrow \H_{Z}(A)\w=\w$ for any $\w$, it follows that
\beas
\v + \u_1-\u_2
=
\H_{Z}(A) \v + \H_{Z}(A) (\u_1-\u_2) 
=
\v + \H_{Z}(A) (\u_1-\u_2).
\eeas
Hence we have $\H_{Z}(A) (\u_1-\u_2) = \u_1-\u_2$, i.e., $\u_1=\u_2$. Therefore, we have $A\v=0$.
\end{proof}

The last condition (\ref{condition}) is verifiable before we conduct the algorithm. As a result, we obtain the following theorem, whose proof is the same as that of \cite[Theorems 3.2, 3.9 and 3.15]{shao2021deterministic}.

\begin{thm}
\label{thm:0726}
Assume the linear system $A\x=\b$ is consistent.
Fix a partition $Z=\{Z_1,\ldots,Z_p\}$ of rows of $A$. Let
\be
\S_0 := \{\x_0, \x_{2k}, \x_{4k},\ldots\}, \quad
\S_1 := \{\x_k, \x_{3k}, \x_{5k},\ldots\}.
\ee
Let $\c_0, \c_1$ be the centers of the minimal sphere supporting $\S_0, \S_1$ respectively. If \eqref{condition} holds, then $\c_0+\c_1 = 2\x_*$.
\end{thm}

As a result, we re-obtain the following well-known result. It is obtained by taking $p=1$ and $Z_1=A$ in (\ref{deterministic block algorithm}).

\begin{prop}
Assume that $A\in \mathbb{R}^{m\times n}$ has full row rank and $\b\in \mathbb{R}^m$. Let $\x_0\in \mathbb{R}^n$ and $\x_1=\H_A \x_0 +2A(AA^\TT)^+ \b$, then $A(\x_0+\x_1)/2 = \b$. Consequently, for any $\x_0$, we have that $\x_0-A(AA^\TT)^+A\x_0 + A(AA^\TT)^+\b$ is a solution of $A\x=\b$.
\end{prop}

\begin{proof}
Since $A$ has full row rank, we have $\H_A = I_n - 2A^+A$ and $A\H_A=-A$. Now it is easy to check the claimed result.
\end{proof}

At the end of this part, we give some results on (\ref{condition}). The results below are comparable to \cite[Proposition 3.11 and Corollary 3.12]{shao2021deterministic}.

\begin{prop}
\label{full row rank case}
If $A$ has full row rank, then \eqref{condition} holds.
\end{prop}

\begin{proof}
By Lemma \ref{lem:generalization of BW formula}, $\H_{Z}(A) = I_n - 2A^\TT X^{-1} A$. If $\H_{Z}(A) \v = \v$, then $A^\TT X^{-1} A\v=0$. Since $A$ has full row rank, $X$ is invertible. As a result, $A\v=0$.
\end{proof}

Based on this, we below provide an effective approach to check (\ref{condition}) for general matrices.
 
\begin{fact}[Sylvester's determinant theorem (also known as Weinstein-Aronszajn identity)]
\label{Sylvester's determinant theorem}
For any two matrices $A,B$ we have $\det(I+AB)=\det(I+BA)$.
\end{fact}

\begin{prop}
\label{general case}
Suppose
$A = \begin{pmatrix}
    A_1 \\
    A_2
\end{pmatrix}\in \mathbb{R}^{m\times n}$ and ${\rm Rank}(A)={\rm Rank}(A_1)$. Assume that $A_1$ has full row rank.
Let $A_2=BA_1$ and let $X_i$ be the block-partition matrix of $A_i$ for $i=1,2$. Then \eqref{condition} holds if and only if $I - B X_1 B^\TT X_2^{+}$ is invertible. In particular, if $X_2$ is invertible, then \eqref{condition} holds if and only if $X_2 - B X_1 B^\TT$ is invertible.
\end{prop}

\begin{proof}
By Proposition \ref{full row rank case},  we have $\H({A_1}) \v = \v \Leftrightarrow A_1\v=0$. Since $A_2$ lies in the row space of $A_1$, we know that $\H({A_1}) \v = \v \Rightarrow \H({A})\v=\v$. Indeed, $\H({A_1}) \v=\v \Rightarrow A_1\v=0 \Rightarrow A_2\v=0 \Rightarrow A\v=0 \Rightarrow \H({A})\v=\v$. Thus (\ref{condition}) holds if and only if $\H({A}) \v = \v \Rightarrow \H({A_1}) \v = \v$. 

Since $A_2$ has full row rank, $X_1$ is invertible.
Suppose $\v = \H({A})\v = \H({A_2}) \H({A_1})\v$, then $\H({A_2})\v = \H({A_1})\v$, i.e.,
\[
(I - 2A_2^\TT X_2^{+} A_2) \v = 
(I - 2A_1^\TT X_1^{-1} A_1) \v. 
\]
So we have
\be
\label{nonsingular condition}
A_1^\TT X_1^{-1} A_1 \v =
A_2^\TT X_2^{+} A_2 \v =
A_1^\TT B^\TT X_2^{+} B A_1 \v,
\ee
i.e.,
\[
A_1^\TT (X_1^{-1} 
- B^\TT X_2^{+} B )A_1 \v = 0.
\]
Since $A_1$ has full row rank, we have $0=(X_1^{-1} 
- B^\TT X_2^{+} B )A_1 \v = X_1^{-1} (I
- X_1 B^\TT X_2^{+} B )A_1 \v$. Thus a necessary and sufficient condition for $A_1\v=0$ is $\det(I
- X_1 B^\TT X_2^{+} B) \neq 0$. 
Indeed, if $\det(I
- X_1 B^\TT X_2^{+} B) = 0$, then there is a $\w\neq 0$ such that $(I
- X_1 B^\TT X_2^{+} B) \w =0$. Since $A_1$ has full row rank, there is a $\v\neq 0$ such that $A_1\v=\w\neq 0$. A contradiction.

By Sylvester's determinant theorem (see Fact \ref{Sylvester's determinant theorem}), we have
\[
\det(I - X_1 B^\TT X_2^{+} B) = \det(I - B X_1 B^\TT X_2^{+}).
\]
If $X_2$ is invertible, then it is further equal to $\det(X_2 - B X_1 B^\TT ) \det(X_2^{-1})$.
\end{proof}

We below provide more characterisations about the matrix $X_2 - B X_1 B^\TT$.

\begin{prop}
\label{prop:some neccessary condition}
Making the same assumptions as Proposition \ref{general case}. Let $Z=\{Z_{1}, \cdots,Z_{p}\}$ be a row partition of $A_2$. 

\begin{enumerate}
\item 
Assume that \eqref{condition} holds, then $\sum_{i=1}^{p} {\rm Rank}(Z_{i})$ is even. In particular, if $X_2$ is invertible and $d$ is the number of rows of $A_2$, then $d$ is even, or equivalently, if $r={\rm Rank}(A)$, then $m-r$ is even.
\item The matrix $X_2 - B X_1 B^\TT$ is anti-symmetric, and its $(i,j)$-th block  is
\[
\begin{cases} \vspace{.1cm}
\ds \frac{1}{2} B_i (X_1^{\T} - X_1) B_j^{\T} & \text{if } i=j, \\
B_i X_1^{\T} B_j^{\T} & \text{if } i < j, \\
-B_i X_1 B_j^{\T} & \text{if } i > j,
\end{cases}
\]
where $1\leq i,j \leq p$.
\end{enumerate}
\end{prop}

\begin{proof}
We shall use the notation appeared in the proof of Proposition \ref{general case}.

(1). Denote $D=\sum_{i=1}^{p} {\rm Rank}(Z_{i})$. Note that $$\det(\H(A)) = \det(\H(A_1)) \det(\H(A_2)) = (-1)^D \det(\H(A_1)).$$ Since $\H(A), \H(A_1)$ are unitary matrices, the complex eigenvalues appear in conjugate forms. Following the proof of  Proposition \ref{general case}, to ensure that $\H(A)$ and $\H(A_1)$ have the same eigenspace of eigenvalue 1, $D$ must be even. When $X_2$ is invertible, we have $d=D$.

(2). Note that $A_2=B A_1$, or more precisely in block form
\[
A_2 = \begin{pmatrix}
    Z_{1} \\
    \vdots \\
    Z_{p}
\end{pmatrix}
=\begin{pmatrix}
    B_{1}A_1 \\
    \vdots \\
    B_{p}A_1
\end{pmatrix}.
\]
So $X_2+X_2^\TT = 2 A_2 A_2^\TT = 2 B A_1 A_1^\TT B^\TT = B (X_1+X_1^\TT) B^\TT$. The $(i,j)$-th block of $X_2$ is
\[
\begin{cases} \vspace{.1cm}
\ds \frac{1}{2} B_i (X_1+X_1^{\T}) B_j^{\T} & \text{if } i=j, \\
B_i (X_1+X_1^{\T}) B_j^{\T} & \text{if } i< j.
\end{cases}
\]
Hence, the $(i,j)$-th block of $X_2 - B X_1 B^\TT$ is
\[
\begin{cases} \vspace{.1cm}
\ds \frac{1}{2} B_i (X_1^{\T} - X_1) B_j^{\T} & \text{if } i=j, \\
B_i X_1^{\T} B_j^{\T} & \text{if } i < j, \\
-B_i X_1 B_j^{\T} & \text{if } i > j.
\end{cases}
\]
This means that $X_2 - B X_1 B^\TT$ is anti-symmetric.
\end{proof}

In practice, we do not have $A_1,A_2$  available in advance. Generally, we start from a row partition $Z$, then run the algorithm. So what kind of $Z=\{Z_1,\ldots,Z_p\}$ should we choose? From Proposition \ref{general case}, one choice is to find an $A_1$ of full row rank from $Z$. But what if this is not possible for a given $Z$. In the algorithm, it is required to compute pseudoinverses of some matrices, so it is preferable to keep the matrices small, i.e., each partition contains a few rows of $A$. With this intuition, one natural choice is that each $Z_i$ has at most $\text{Rank}(A)$ rows. From Proposition \ref{prop:some neccessary condition}, it is necessary to make sure that $m-\text{Rank}(A)$ is even. In summary, we provide the following simple criterion to run the algorithm.

\begin{quote}
{\bf The criterion:} Let $Z=\{Z_1,\ldots,Z_p\}$ be a row partition. To use this algorithm, it is preferable to make sure that each $Z_i$ has full-row rank and $m-\text{Rank}(A)$ is even.
\end{quote}
Both conditions in the above criterion are easy to satisfy. As for the second condition, we can introduce an extra trivial row if necessary. But this idea is not always working from the proof of Proposition \ref{general case}. There are some nonsingularity conditions for it works. From (\ref{nonsingular condition}), a general one is that we randomly decompose $Z$ into two parts $A_1$ and $A_2$, then $A_1^\TT X_1^{-1} A_1 -
A_2^\TT X_2^{+} A_2$ should be nonsingular.

If the linear system is consistent and $A$ has full column rank, then 
$\H_A = -I$. Thus in the iteration (\ref{deterministic block algorithm}), if $i=1$ and $Z=A$, then  $\x_{k+1} - \x_* = -(\x_k - \x_*)$, where $\x_*$ is the solution of $A\x=\b$. Namely, during the iteration, there are only two points that are symmetric with respect to the solution. More generally, given a partition $\{Z_1,\ldots,Z_p\}$ of the rows of $A$, if $Z_i$ for some $i$ has full column rank, then $\H_{Z_i}=-I$. Even this, we can not delete it in the iteration since we need to make sure that $m-{\rm Rank}(A)$ is even by Proposition \ref{prop:some neccessary condition}.

\subsection{Properties of $\eta_Z(A)$}

In this subsection, we study more properties of $\eta_Z(A)$, which is relevant to the convergence rate due to Theorem \ref{thm:convergence}. 
We shall mainly consider the case when each $Z_i$ is a row of $A$. Now the definition can be restated as follows. It is a little different from \cite[Definition 1.2]{shao2021deterministic}, but essentially
the same thing.

\begin{defn}
\label{defn:eta1}
Let $A$ be an $m\times n$ matrix such that all rows are nonzero. Let $\R_i$ be the reflection generated by the $i$-th row of $A$, and let $\R_A=\R_n \cdots \R_1$. Denote the eigenvalues of $\R_A$ be $e^{2i\theta_1},\ldots,e^{2i\theta_m}$ with $\theta_j\in[0,\pi)$ for all $j$. Define
\be
\eta(A) := \frac{1}{\min_{\ell:\theta_\ell\neq 0} |\sin(\theta_\ell)|}.
\ee
\end{defn}

In \cite[Definition 1.2]{shao2021deterministic}, $\eta(A)$ is defined by $\frac{1}{\min_{\ell:\theta_\ell\neq 0} |\theta_\ell|}.$ Here we feel $\sin(\theta_l)$ has more clear geometric meaning. A reason for this change will be clear shortly. But theoretically, when $\theta_\ell$ is small, these two quantities are close to each other.

The following results are easy to check through the definition of $\eta(A)$.

\begin{prop}
\label{prop: some facts}
We have the following properties about $\eta(A):$

\begin{itemize}
\item Let $Q$ be an orthogonal matrix, then $\R_{AQ} = Q^\TT \R_A Q$, so $\eta(A)$ is invariant under the right action of orthogonal matrices.

\item Let $D$ be an invertible diagonal matrix, then $\eta(DA)=\eta(A)$, so  $\eta(A)$ is invariant under row scaling.

\end{itemize}
\end{prop}

\begin{lem}[Lemma 2.2 of \cite{shao2021deterministic}, also a special case of Lemma \ref{lem:generalization of BW formula}]
\label{cited lemma}
Let $W$ be the lower-triangular matrix satisfying $W+W^\T=2AA^\T$, then $\R_A = I - 2A^\TT W^{-1} A$.
\end{lem}

\begin{thm}
\label{thm for eta}
Let $A$ be an $m\times n$ matrix with no zero rows and let its $i$-th row as $A_i$. Denote $c_{ij}=(A_i\cdot A_j)/\|A_i\|\|A_j\|$.
Then the eigenvalues of $\R_A$ are determined by the following equation
\be
\label{char1}
\det\begin{pmatrix}
1+x & 2c_{12} & \cdots & 2c_{1m} \\
2c_{12}x & 1+x & \cdots  & 2c_{2m} \\
\vdots & \vdots & \ddots & \vdots \\ 
2c_{1m}x & 2c_{2m}x & \cdots  & 1+x
\end{pmatrix}
=0.
\ee
\end{thm}

\begin{proof}
Let $(x, \u)$ be an eigenpair of $\R_A = I - 2A^\TT W^{-1} A$, then we have $2A^\TT W^{-1} A\u = (1-x) \u$. Multiplying $A$ on both sides gives rise to $(W+W^\TT)W^{-1} A\u = (1-x) A\u$, i.e., $(W^\TT W^{-1} +x I) A\u = 0$. 
If $A\u \neq 0$, then
$\det(W^\TT W^{-1} +x I)=0$. Equivalently, $\det(W^\TT  +x W)=0$. 
If $A\u = 0$, then $x=1$. We also have $\det(W^\TT+W)=\det(2AA^\T)=0$. The multiplicity of this eigenvalue $x=1$ is at least $m-\text{Rank}(A)$. Without considering multiplicity, the eigenvalues are the same. In matrix form, we have (\ref{char1}).
\end{proof}

Theorem \ref{thm for eta} was indeed already proved by Coxeter \cite{coxeter1939osung} in 1939. But this paper is very old and hard to find. A statement of the result without a proof was given in another paper of Coxeter \cite{coxeter1947product}. 

When computing condition numbers, we usually consider the eigenvalues of $AA^\TT$. We now assume that all rows of $A$ are normalised to have norm 1. In matrix form, the characteristic polynomial of $AA^\TT$ is
\be
\label{char2}
\det\begin{pmatrix}
1-x & c_{12} & \cdots & c_{1m} \\
c_{12} & 1-x & \cdots  & c_{2m} \\
\vdots & \vdots & \ddots & \vdots \\ 
c_{1m} & c_{2m} & \cdots  & 1-x
\end{pmatrix}
=0.
\ee
From (\ref{char1}) and (\ref{char2}) we can see that if $\eta(A)$ is large, i.e., when $x = e^{2i\theta} \approx 1$ is a solution of (\ref{char1}), then $(1-x)/2 \approx \sin^2(\theta)$ is an approximate solution of (\ref{char2}). The converse also holds. As a result, $\eta(A) \approx 1/\sigma_{\min}(A)$, where $\sigma_{\min}$ refers to the minimal nonzero singular value of $A$. 

Note that in the general case, the singular values are described by (\ref{char2}) with diagonal entries replaced with $1- \frac{x}{\|A_1\|^2}, 1- \frac{x}{\|A_2\|^2}, \ldots, 1- \frac{x}{\|A_m\|^2}$. In comparison, $\eta(A)$ is always described by (\ref{char1}) because of Proposition \ref{prop: some facts}.

In \cite[Propositions 3.5 and 3.6]{shao2021deterministic}, we proved that $\eta(A) = O(\kappa(A)^2 \log m)$ for all $A$ and $\eta(A)\approx \kappa(A)$ for random matrices $A$. Here $\kappa(A)$ is the condition number of $A$. 
We below give a more quantitative characterisation between $\eta$ and $\sigma_{\min}^{-1}$.

\begin{thm}
\label{thm:eta vs sigma-min}
Let $L$ be the lower triangular matrix such that $L+L^\T = AA^\T - I$. Then we have
\be
\label{thm:eta vs sigma-min-eq}
\frac{|\eta^2-\sigma^{-2}_{\min}(A)|^2}{\eta^2\sigma^{-2}_{\min}(A)
\Big(1 + \eta^{2}\Big)\Big(1+\sigma^{-2}_{\min}(A)\Big)}
 \leq 4\|L\|^2.
\ee

\end{thm}

\begin{proof}
We will use Proposition \ref{prop:gep2}. The singular values of $A$ are described by the generalised eigenvalue problem $AA^\T \u = y \u$. In (\ref{char1}), let $x=1-2y$, then we obtain $\det(AA^\T - y (I+2L))=0$. So $y$ can be described by the generalised eigenvalue problem $AA^\T \u = y (I+2L) \u$. Using the notation of Proposition \ref{prop:gep2}, we have $Z=(AA^\T, I)$ and $W=(AA^\T, I+2L)$. By that proposition, we have
\[
S_Z(W) \leq d_2(Z,W)
\leq \|(I+(AA^\T)^2)^{-1/2}\| \|2L\|
\leq 2 \|L\|.
\]
The first inequality is because $AA^\T$ is symmetric so that $T$ in Proposition \ref{prop:gep2} is orthogonal. The second inequality follows from (\ref{upper bound of d2}). The last inequality is because $\|(I+(AA^\T)^2)^{-1/2}\| = (1+\sigma_{\min}(A)^4)^{-1/2} \leq 1$. 

From the definition of $S_Z(W)$, we have that for any $y$ satisfying $AA^\T \u = y (I+2L) \u$,
\be
\label{0604}
\frac{|y -\sigma^2_{\min}(A)|}{
\sqrt{\Big(1 + |y|^{2}\Big)\Big(1+\sigma^{2}_{\min}(A)\Big)}}
\leq S_Z(W) \leq 2\|L\|.
\ee
Now assume that $\eta:=\eta(A) = 1/\sin(\theta)$ for some $\theta > 0$, i.e., there is a $x=e^{2i\theta}$ satisfying (\ref{char1}). Then
\[
y = \frac{1-x}{2} = \sin^2(\theta)-i\sin(\theta)\cos(\theta) 
= \eta^{-2} - i \eta^{-1} \sqrt{1 - \eta^{-2}}.
\]
So $| y - \sigma^2_{\min}(A)|
\geq |\eta^{-2} - \sigma^2_{\min}(A)| = \frac{|\eta^2-\sigma^{-2}_{\min}(A)|}{\eta^2\sigma^{-2}_{\min}(A)}$ and $|y|^2=\eta^{-2}$.
Substituting these into (\ref{0604}) yields
\[
4\|L\|^2 \geq \frac{|\eta^2-\sigma^{-2}_{\min}(A)|^2}{\eta^4\sigma^{-4}_{\min}(A)
\Big(1 + \eta^{-2}\Big)\Big(1+\sigma^{2}_{\min}(A)\Big)}
=
\frac{|\eta^2-\sigma^{-2}_{\min}(A)|^2}{\eta^2\sigma^{-2}_{\min}(A)
\Big(1 + \eta^{2}\Big)\Big(1+\sigma^{-2}_{\min}(A)\Big)}.
\]
This completes the proof.
\end{proof}

From the above theorem, we can see that if $\|L\|$ is close to 0, i.e., when $A$ is close to orthogonal, then $\eta \approx \sigma^{-1}_{\min}(A)$. We can make this more precise as follows.

\begin{defn}[restricted isometry property \cite{candes2005decoding}]
Let $\delta\in(0,1)$ and $s\in \mathbb{N}$.
A matrix $A\in \mathbb{C}^{m\times n}$ is called to satisfy the $s$-restricted isometry property with restricted isometry constant $\delta$ if for
every submatrix $A_s \in \mathbb{C}^{m\times s}$,  we have
\[
(1-\delta) \|\x\|^2 \leq \|A_s \x\|^2 
\leq 
(1+\delta) \|\x\|^2
\]
for every $\x\in \mathbb{C}^s$.
Namely, $\|A_s^\dag A_s - I_s\|_{2\rightarrow 2} \leq \delta$, or equivalently all the eigenvalues of $A_s^\dag A_s$ lie in the interval $[1-\delta, 1+\delta]$. Here $A_s^\dag$ is the conjugate transpose of $A_s$.
\end{defn}

In randomised numerical linear algebra, Johnson-Lindenstrauss lemma shows that there is a linear map $f:\mathbb{R}^{N}\rightarrow \mathbb {R}^{n}$ satisfying that for any set $X$ of $m$ points in $\mathbb{R}^N$ and integer $n>8 (\log m)/\delta^2$, we have $(1-\delta) \|u-v\|^2 \leq \|f(u)-f(v)\|^2 \leq (1+\delta) \|u-v\|^2$ for every $u,v \in X$ with high probability, e.g., see \cite{woodruff2014sketching}. The map $f$ plays an important in randomised numerical linear algebra that can reduce a large scale matrix problem to a small scale but ``equivalent" matrix problem with a tolerable error. Clearly, this map has the restricted isometry property in the randomised sense.

\begin{lem}[Theorem 1 of \cite{angelos1992triangular}]
\label{lem:lower triangular part}
Let $M_\ell$ be the lower triangular part of $M\in \mathbb{C}^{m\times m}$, then $\|M_\ell\| = O(\|M\|\log m)$ for every $m\geq 2$.
\end{lem}

By Theorem \ref{thm:eta vs sigma-min} and Lemma \ref{lem:lower triangular part} with $M=A^\dag A-I$, we have the following result.

\begin{cor}
\label{cor:restricted isometry property}
Assume that $A\in \mathbb{R}^{m\times n}$ has $n$-restricted isometry property with restricted isometry constant $\delta$, then 
\be
\frac{|\eta^2-\sigma^{-2}_{\min}(A)|^2}{\eta^2\sigma^{-2}_{\min}(A)
\Big(1 + \eta^{2}\Big)\Big(1+\sigma^{-2}_{\min}(A)\Big)}
\leq O(\delta^2 \log^2 m).
\ee
\end{cor}



As shown in Proposition \ref{prop: some facts}, $\eta(A)$ is invariant under the right action of orthogonal matrices. Usually, it is not invariant under left action of orthogonal matrices. We describe this change as follows.

\begin{thm}
\label{thm: some facts 2}
Assume that $A\in \mathbb{R}^{m\times n}$.
Let $P$ be an orthogonal matrix of dimension $m$, then 
\be
\frac{|\eta(PA)^2-\eta(A)^2|}{\eta(PA) \eta(A) \sqrt{(1+\eta(A)^2)(1+\eta(PA)^2)}}
\leq O\Big( \kappa(A)^2 \|P^\T L_{PA} P-L_A\|^{1/m} \|A\|^{-2/m} \Big),
\ee
where $L_A$,  $L_{PA}$ are lower triangular matrices such that $L_A+L_A^\T = AA^\T - I$ and $L_{PA}+L_{PA}^\T = PAA^\T P^\T - I$ respectively.
\end{thm}

\begin{proof}
From the proof of Theorem \ref{thm:eta vs sigma-min}, we know that the eigenvalues of $\R_A$ are described by the generalised eigenvalue problem $AA^\T \u = y (I+2L_A) \u$. Similarly, the eigenvalues of $\R_{PA}$ is described by the generalised eigenvalue problem $P AA^\T P^\T \u = y (I+2L_{PA}) \u$. Namely, they are described by $ AA^\T ( P^\T \u) = y (I+2 P^\T L_{PA} P)  ( P^\T \u)$. By Proposition \ref{prop:gep1} with $Z = (AA^\T , I+2L_A)$ and $W = (AA^\T , I+2 P^\T L_{PA} P)$, we have
\[
S_Z(W) \leq 
\frac{\sqrt{\|A\|^4 + \|I+2L_A\|^2}}{\gamma(AA^\T , I+2L_A)} \left( \frac{\sqrt{2} \|P^\T L_{PA} P - L_A\|}{\sqrt{\|A\|^4 + \|I+2L_A\|^2}} \right)^{1/m}.
\]
By Remark \ref{remark: on gamma}, $\gamma(AA^\T , I+2L_A) \geq \sigma_{\min}(AA^\T)=\sigma_{\min}(A)^2$. By Lemma \ref{lem:lower triangular part}, $\|L_A\|\leq \|A\|^2 + 1$. So
\[
\frac{\sqrt{\|A\|^4 + \|I+2L_A\|^2}}{\gamma(AA^\T , I+2L_A)}
\leq O(\kappa_A^2).
\]
Similar to the argument used in the proof of Theorem \ref{thm:eta vs sigma-min}, we have that
\[
\frac{|\eta(A)^{-2}-\eta(PA)^{-2}|}{\sqrt{1+\eta(A)^{-2}} \sqrt{1+\eta(PA)^{-2}}} \leq S_Z(W).
\]
Simplifying this and putting the above all together lead to the claimed result.
\end{proof}

As singular values do not change after orthogonal operations. So from Theorem \ref{thm:eta vs sigma-min} and the triangle inequality, we have another description of the change described as follows:
\[
\frac{\Big|\eta(PA)^2 - \eta(A)^2\Big|^2}{\max\Big(\eta(A)^2,\eta(PA)^2\Big) \sigma^{-2}_{\min}(A)
\Big(1 + \max\Big(\eta(A)^2,\eta(PA)^2\Big) \Big)\Big(1+\sigma^{-2}_{\min}(A)\Big)}
 \leq 8(\|L_A\|^2+\|L_{PA}\|^2).
\]
Similar to Corollary \ref{cor:restricted isometry property}, for matrices with restricted isometry property, the above bound on the right hand side can be small, and so $\eta(A)$ is close to invariant under left actions of orthogonal matrices.



Proposition \ref{prop: some facts} also holds for $\eta_Z(A)$. Regarding Theorem \ref{thm for eta}, we similarly have the following characterisation, whose proof is similar due to Lemma \ref{lem:generalization of BW formula}.

\begin{thm}
\label{thm for etaZ}
Let $A$ be an $m\times n$ matrix with row partition $Z=\{Z_1,\ldots,Z_p\}$ such that all $Z_i$ has full row rank.
Denote $C_{ij}=(Z_iZ_i^\T)^{-1/2}(Z_i Z_j^\T) (Z_jZ_j^\T)^{-1/2}$.
Then the eigenvalues of $\H_Z(A)$ are determined by the following equation
\be
\det\begin{pmatrix} \vspace{.1cm}
(1+x)I & 2 C_{12} & \cdots & 2 C_{1p} \\ \vspace{.1cm}
2 C_{21} x & (1+x)I & \cdots & 2 C_{2p} \\ \vspace{.1cm}
\vdots & \vdots & \ddots & \vdots \\
2 C_{p1} x & 2C_{p2}x & \cdots & (1+x)I
\end{pmatrix}
=0.
\ee
\end{thm}

Let
\[
B = \begin{pmatrix}
(Z_1Z_1^\T)^{-1/2}Z_1 \\
(Z_2Z_2^\T)^{-1/2}Z_2 \\
\vdots \\
(Z_pZ_p^\T)^{-1/2}Z_p \\
\end{pmatrix},
\]
then
\[
BB^\T - x I = 
\det\begin{pmatrix} \vspace{.1cm}
(1-x)I & C_{12} & \cdots & C_{1p} \\ \vspace{.1cm}
C_{21}  & (1-x)I & \cdots & C_{2p} \\ \vspace{.1cm}
\vdots & \vdots & \ddots & \vdots \\
C_{p1}  & C_{p2} & \cdots & (1-x)I
\end{pmatrix}.
\]
From the above analysis, it is not hard to imagine that $\eta_Z(A)$ relates to the minimal singular value of $B$. We will not study this in more detail here.


\subsection{Connections of the solution of least squares}

In \cite{shao2021deterministic}, it was shown that the deterministic Kaczmarz method of the scheme (\ref{standard verison}), where $i_k$ is chosen cyclically, cannot be used to solve least squares. The iterative process indeed solves 
$A^\T W^{-1} A \x = A^\T W^{-1} \b$, where $W$ is the lower triangular matrix such that $W + W^\T = AA^\T - I$.
In this section, we analyse in more detail why this is the case and present a method to fix this. Some notation were introduced in \cite[Section 4]{shao2021deterministic}.

\begin{prop}
Let $A\in \mathbb{R}^{m\times n}$ and
\[
F := 2\sum_{j=1}^m \frac{1}{\|A_j\|^2} \R_m \cdots \R_{j+1} A_j \e_j^\TT,
\]
then $F = 2A^\TT W^{-1}$, where $W$ is the lower triangular matrix satisfying $W+W^\TT = 2 AA^\TT$. Thus $FA=I-\R_A$.
\end{prop}

\begin{proof} 
It suffices to prove $FW=2A^\TT$. By definition,
\beas
FW &=& 2\sum_{j=1}^m \frac{1}{\|A_j\|^2} \R_m \cdots \R_{j+1} A_j \left(\sum_{k=1}^{j-1} 2 A_j^\TT A_k \e_k^\TT + \|A_j\|^2 \e_j^\TT\right) \\
&=& 2\sum_{k=1}^{m-1}  \left(\sum_{j=k+1}^m \frac{2}{\|A_j\|^2} \R_m \cdots \R_{j+1} A_jA_j^\TT \right)
A_k \e_k^\TT + 2\sum_{j=1}^m \R_m \cdots \R_{j+1} A_j \e_j^\TT \\
&=& 2\sum_{k=1}^{m-1}  \left(\sum_{j=k+1}^m  \R_m \cdots \R_{j+1} (I_n -\R_j) \right) A_k \e_k^\TT
+2\sum_{j=1}^m \R_m \cdots \R_{j+1} A_j \e_j^\TT \\
&=& 2\sum_{k=1}^{m-1}  \left( I_n - \R_m \cdots \R_{k+1} \right) A_k \e_k^\TT
+2\sum_{j=1}^m \R_m \cdots \R_{j+1} A_j \e_j^\TT \\
&=& 2A^\TT.
\eeas
This completes the proof.
\end{proof}

For any $i\in\{0,1,\ldots,m-1\}$, define $i$-th shifted product of reflections as
\be
\R_A^{(i)} := \R_{i} \cdots \R_1 \R_m \R_{m-1} \cdots \R_{i+1},
\ee
where $\R_{i} \cdots \R_1$ is viewed as the identity matrix when $i=0$. So $\R_A=\R_A^{(0)}$. we also denote
\[
F_A^{(i)} =  \sum_{j=1}^m \frac{2}{\|A_{j+i}\|^2} \R_{m+i}\R_{m+i-1} \cdots \R_{j+i+1} A_{j+i} \e_j^\TT,
\]
and the shift matrix as
\[
Q = \begin{pmatrix}
0 & 1 & \cdots & 0 \\
\vdots & \vdots & \ddots & \vdots \\
0 & 0 & \cdots & 1  \\
1 & 0 & \cdots & 0  \\
\end{pmatrix}_{m\times m}.
\]

For convenience, we now assume that $m-n$ is even and ${\rm Rank}(A) = n$.
For the linear system $A\x=\b$, we start from an $\x_0$, then iterate:
\be
\label{iterative process 718}
\x_{k+1} = \x_k + 2\frac{b_{i_k} - A_{i_k}^{\T} \x_k}{\|A_{i_k}\|^2} A_{i_k} = \R_{i_k} \x_k + 2\frac{b_{i_k}}{\|A_{i_k}\|^2} A_{i_k},
\ee
where $i_k = (k \mod m) + 1$. We decompose $\{\x_0,\x_1,\x_2,\cdots\}$ into $m$ parts:
\[
\S_i = \{\x_{i+km}: k=0,1,2,\ldots,\},
\quad i\in\{0,1,\ldots,m-1\}.
\]
Then each $\S_i$ lies on a high dimensional sphere, and according to \cite[Theorem 4.2]{shao2021deterministic} its centre is the solution of the linear system
\be
F_A^{(i)} Q^{i} A \x = F_A^{(i)} Q^{i} \b.
\ee

\begin{prop}
\label{prop:relation of solutions}
Assume that $m-n$ is even and ${\rm Rank}(A)=n$. Let $\x_*^{(i)}$ be the solution of $F_A^{(i)} Q^{i} A \x = F_A^{(i)} Q^{i} \b$.
Then for any $i\geq 1$
\be
\label{relation of solutions}
\x_*^{(i)} = \R_i \x_*^{(i-1)} + \frac{2b_{i}}{\|A_{i}\|^2} A_i.
\ee
\end{prop}

\begin{proof}
We first consider the case $i=1$. Now we have
\[
F^{(1)} 
= \sum_{j=1}^{m-1} \frac{2}{\|A_{j+1}\|^2} \R_{1}\R_m\cdots \R_{j+2} A_{j+1} \e_j^\TT+ \frac{2}{\|A_{1}\|^2} A_1 \e_m^\TT.
\]
We also have $F^{(1)}_A {QA} = I - \R_{QA} = \R_1 (I - \R_A) \R_1^\TT$. 
The equation for $i=1$ is
\[
(I - \R_{QA})\x_*^{(1)} = \sum_{j=1}^{m-1} \frac{2b_{j+1}}{\|A_{j+1}\|^2} \R_{1}\R_m\cdots \R_{j+2} A_{j+1} + \frac{2b_{1}}{\|A_{1}\|^2} A_1.
\]
Namely,
\beas
(I - \R_A) \R_1^\TT \x_*^{(1)} &=& \sum_{j=1}^{m-1} \frac{2b_{j+1}}{\|A_{j+1}\|^2} \R_m\cdots \R_{j+2} A_{j+1} - \frac{2b_{1}}{\|A_{1}\|^2} A_1 \\
&=& F^{(0)}\b - \frac{2b_{1}}{\|A_{1}\|^2} (I-\R_A)A_1.
\eeas
So $(I - \R_A) (\R_1^\TT \x_*^{(1)}+\frac{2b_{1}}{\|A_{1}\|^2}A_1) = F^{(0)}\b$, i.e., $\x_*^{(0)} = \R_1^\TT \x^{(1)}+\frac{2b_{1}}{\|A_{1}\|^2}A_1$, or equivalently $\x_*^{(1)} = \R_1 \x_*^{(0)} + \frac{2b_{1}}{\|A_{1}\|^2} A_1$. In this step, we used the facts that $m-n$ is even and ${\rm Rank}(A)=n$ due to \cite[Proposition 4.5]{shao2021deterministic}.
Similarly, we have $\x_*^{(i)} = \R_i \x_*^{(i-1)} + \frac{2b_{i}}{\|A_{i}\|^2} A_i$ for any $i\geq 1$.
\end{proof}

The connection (\ref{relation of solutions}) between $ \x_*^{(i)}$ and $\x_*^{(i-1)}$ is similar to the iterative procedure (\ref{iterative process 718}).
The above result also implies that for any $i\geq 1$
\[
(I+\R_i) ( \x_*^{(i)} - \x_*^{(i-1)} ) =0 ,
\]
where $\x_*^{(m)}:=\x_*^{(0)}$. So $\x_*^{(i)} - \x_*^{(i-1)} = \alpha_i \ket{A_i}$ for some $\alpha_i$, where $\ket{A_i}=\frac{A_i}{\|A_i\|}$ is the Dirac notation. Therefore,
\be
\label{equality0}
\sum_{i=1}^m \alpha_i \ket{A_i} = 0.
\ee
By (\ref{relation of solutions}), we have 
\beas
\alpha_i = \frac{2b_i}{\|A_i\|} -  \frac{2}{\|A_i\|} A_i^\T  \x_*^{(i-1)} .
\eeas

As a direct application of  Proposition \ref{prop:relation of solutions}, we have the following relation.

\begin{cor}
\label{cor:connection}
Let $\x_*^{(i)}$ be defined as in Proposition \ref{prop:relation of solutions}, then we have
\be
\label{0109eq1}
\sum_{i=1}^{m} \|A_i\|^2 \x_*^{(i)}  = \sum_{i=1}^{m} \|A_i\|^2  \R_i \x_*^{(i-1)} + 2A^\T \b .
\ee
\end{cor}
In comparison, let $\x_{\rm LS}$ the solution of the least square problem $\arg\min\|A\x-\b\|$, i.e., $A^\TT A \x_{\rm LS} = A^\TT \b$, then
\be
\label{0109eq2}
\Big(\sum_{i=1}^m \|A_i\|^2\Big) \x_{\rm LS} = \Big(\sum_{i=1}^m \|A_i\|^2 \R_i\Big) \x_{\rm LS} + 2A^\T \b.
\ee
The difference of the two equations \eqref{0109eq1}, \eqref{0109eq2} explains why we can not use $\frac{1}{m} \sum_{i=1}^m \x_*^{(i)}$ to approximate $\x_{\rm LS}$ directly, i.e., why the iterative process (\ref{standard verison}) with $i_k$ chosen cyclically does not solve the least squares. 
But for consistent linear systems, we have $\x_*^{(1)}=\cdots=\x_*^{(m)}$. In this case, the above two equations coincide with each other.
Note that $\x_*^{(i)}$ is invariant under row scaling, while $\x_{\rm LS}$ does. 

We below present more connections between $\x^{(i)}$ and $\x_{\rm LS}$.

\begin{prop}
\label{prop:connection}
Assume $\|A_i\|=1$ for all $i$, then we have the following properties:
\begin{enumerate}
\item $\sum_{i=1}^{m} \ket{A_i} \langle A_i ,\x_*^{(i)} - \x_{\rm LS} \rangle = 0.$
\item $\x_{\rm LS} 
= (mI_m - \sum_{j=1}^m \R_j)^{-1}
\sum_{i=1}^{m} (I_m-\R_i)\x^{(i-1)}.$
\item If we assume that $\x_{\rm LS} = \sum_{i=1}^m \lambda_i \x_*^{(i)}$, then $\lambda_1,\ldots,\lambda_m$ satisfy
\[
\sum_{i,j=1}^m \lambda_i \Big(\R_j \x_*^{(i)} -\R_i \x_*^{(i-1)} \Big) 
+ 2\sum_{i=1}^m b_i (1-m\lambda_i) A_i = 0.
\]
\end{enumerate}

\end{prop}

\begin{proof}
From (\ref{0109eq1}) and (\ref{0109eq2}), when $\|A_i\|=1$ for all $i$, we have
\bes
\sum_{i=1}^{m} (I-\R_i) \x_*^{(i-1)}  = 2A^\T \b =
\sum_{i=1}^{m} (I-\R_i) \x_{\rm LS} .
\ees
So
\be \label{equality1}
\sum_{i=1}^{m} \ket{A_i} \langle A_i ,\x_*^{(i-1)} - \x_{\rm LS} \rangle = 0.
\ee
This is indeed (\ref{equality0}), but usually $\alpha_i \neq \langle A_i ,\x_*^{(i-1)} - \x_{\rm LS} \rangle$. 
  By (\ref{equality0}) and (\ref{equality1}), we also have
\be
\sum_{i=1}^{m} \ket{A_i} \langle A_i ,\x_*^{(i)} - \x_{\rm LS} \rangle = 0.
\ee
From (\ref{equality1}), we also obtain
\bes
\x_{\rm LS} 
= (mI - \sum_{j=1}^m \R_j)^{-1}
\sum_{i=1}^{m} (I-\R_i)\x_*^{(i-1)} .
\ees

Now assume $\x_{\rm LS} = \sum_{i=1}^m \lambda_i \x_*^{(i)}$, then
\[
m \sum_{i=1}^m \lambda_i \x_*^{(i)} = \Big(\sum_{j=1}^m  \R_j \Big) \sum_{i=1}^m \lambda_i \x_*^{(i)} + 2\sum_{i=1}^m b_i A_i.
\]
Namely,
\[
\sum_{i=1}^m \Big(m - \sum_{j=1}^m  \R_j  \Big) \lambda_i \x_*^{(i)} =  2A^\T \b.
\]
By (\ref{relation of solutions}), the LHS is
\[
m \sum_{i=1}^m \lambda_i (\R_i \x_*^{(i-1)} + 2b_{i} A_i).
\]
So we obtain
\[
\sum_{i,j=1}^m \lambda_i \Big(\R_j \x_*^{(i)} -\R_i \x_*^{(i-1)} \Big) 
+ 2\sum_{i=1}^m b_i (1-m\lambda_i) A_i = 0.
\]
This completes the proof.
\end{proof}

\subsection{Some final comments}

In matrix form, we know from (\ref{relation of solutions}) that
\[
\begin{pmatrix}
    \x_*^{(i)} \\
    1
\end{pmatrix} =
\begin{pmatrix}
    \R_i & \frac{2b_{i}}{\|A_{i}\|^2} A_i \\
     0 & 1
\end{pmatrix}
\begin{pmatrix}
    \x_*^{(i-1)} \\
    1
\end{pmatrix}.
\]
This is the same procedure as the Kaczmarz method.
In particular, when $A$ is invertible (or has full row rank), then $\x_*^{(i)}=\x_*^{(i-1)}=A^{-1}\b=:\x_*$, so 
$\begin{pmatrix}
    \x_* \\
    1
\end{pmatrix}$ 
lies in the invariant space of 
\[
\C_i := \begin{pmatrix}
    \R_i & \frac{2b_{i}}{\|A_{i}\|^2} A_i \\
     0 & 1
\end{pmatrix}, \quad i \in [m].
\]
From the above, for any $\omega_1,\ldots,\omega_m>0$, we also have 
\be \label{cimmino}
\sum_{i=1}^m \omega_i \C_i \x_* = \omega \x_*,
\ee
where $\omega=\sum_i \omega_i$. 
Recall that Cimmino's algorithm \cite{Cimmino,benzi2004gianfranco} reads as follows: Arbitrary choose $\c_0\in \mathbb{R}^n$, update
\[
\c_{k+1} = \sum_{i=1}^m \frac{\omega_i}{\omega} \left(\c_k + 2\frac{b_i-A_i^\T \c_k}{\|A_i\|^2}\right).
\]
In our notation, if we set $\x_k = (\c_k,1)^\T$, then the above formula becomes
\[
\x_{k+1} = \sum_{i=1}^m \frac{\omega_i}{\omega}  \C_i \x_k.
\]
This is consistent with (\ref{cimmino}). Cimmino's algorithm aims to find an invariant vector of the operator $\sum_{i=1}^m \frac{\omega_i}{\omega}  \C_i$.

Another way to understand Kaczmarz algorithm is as follows.
Let $P(t) = \C_1 + \C_2 t + \cdots + \C_m t^{m-1} - m t^m$. We say that $\lambda$ is an eigenvalue of $P(t)$ if there exists a nonzero vector $\x$ satisfying $P(\lambda) \x = 0$; $\x$ is then a corresponding eigenvector.  Consider the block companion matrix of $P(t)$:
\[
C_P = \begin{pmatrix}
0 & I \\
  & 0 & I \\
  &   & \ddots & \ddots \\
  &   &        & 0 & I \\
\frac{\C_1}{m} & \frac{\C_2}{m} & \cdots & \frac{\C_{m-1}}{m} & \frac{\C_m}{m}
\end{pmatrix}.
\]
It is easy to check that the eigenpairs $(\lambda, \tilde{\x})$ of $C_P$ must have the form $\tilde{\x} = (\x, \lambda \x, \lambda^2 \x, \cdots, \lambda^{m-1} \x)^\T$ where $P(\lambda) \x = 0$.
So we can recover $\x_*$ if $\lambda$ is a solution of $1+t+\cdots+t^{m-1}-mt^m=0$. Particularly, when $\lambda=1$, we have $\tilde{\x}_* = (\x_*, \x_*, \x_*, \cdots, \x_*)^\T$. If we start from an arbitrary chosen vector $\tilde{\x}_0 = (\x_0, \x_0, \x_0, \cdots, \x_0)^\T$ and repeat the iteration $\tilde{\x}_0 \rightarrow C_P \tilde{\x}_0$, then this process converges to $\tilde{\x}_*$. Here the last entry of $\x_0$ should be set as 1.

\section{Numerical illustration}
\label{section:Numerical demonstration}

In this section, we mainly demonstrate the randomised reflective Kaczmarz algorithms for solving linear systems. Here we will not compare the efficiency of them with previous Kaczmarz algorithms, as some of the comparisons have already been made in \cite{shao2021deterministic}. Our main focus here is to demonstrate the differences of the algorithms presented in this paper. The code, written by Maple 2024, is available via this link.\footnote{code link: \url{https://drive.google.com/file/d/16Qxeh91V9Isxky6FnJ24uUOibkAVbQPd/view?usp=sharing}.} All experiments were conducted on an iMac Pro with processor 3.2 GHz 8-Core Intel Xeon W.

Figure \ref{fig1} exhibits the difference between the reflective Kaczmarz algorithm (\ref{standard verison}) for consistent systems and inconsistent systems. As shown in Figure \ref{fig1-1}, for consistent systems, the points generated by the procedure (\ref{standard verison}) lie on a sphere exactly. This is also guaranteed theoretically. 
For inconsistent systems, as shown in Figure \ref{fig1-2}, the points no longer lie on a sphere but instead are distributed around the solution in a manner similar to spheres.

\begin{figure}
     \centering
     \begin{subfigure}[b]{0.4\textwidth}
         \centering
         \includegraphics[width=\textwidth]{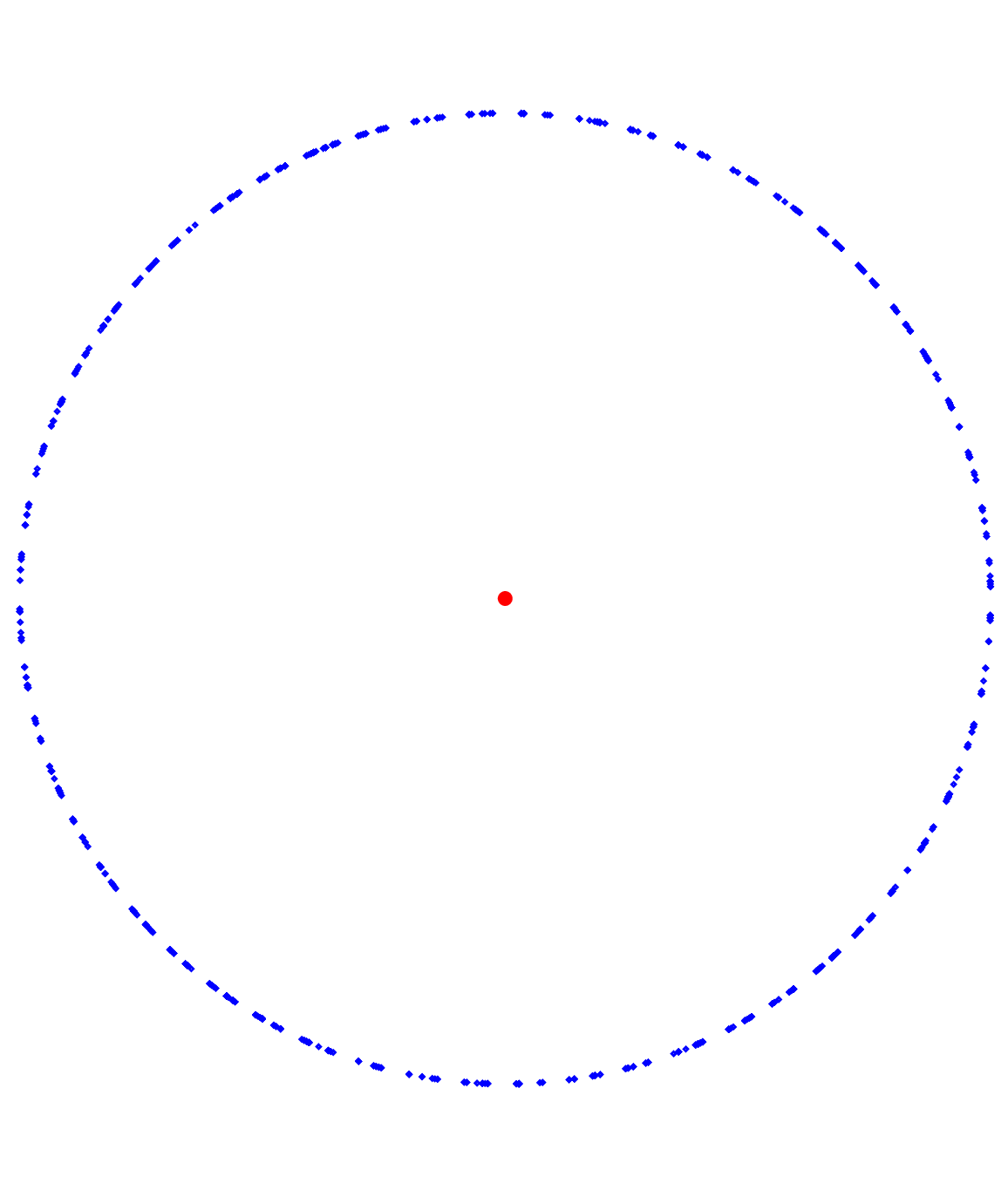}
         \caption{Generate a series of points on a sphere around the solution}
         \label{fig1-1}
     \end{subfigure}
     \hfill     
     \begin{subfigure}[b]{0.4\textwidth}
         \centering
         \includegraphics[width=\textwidth]{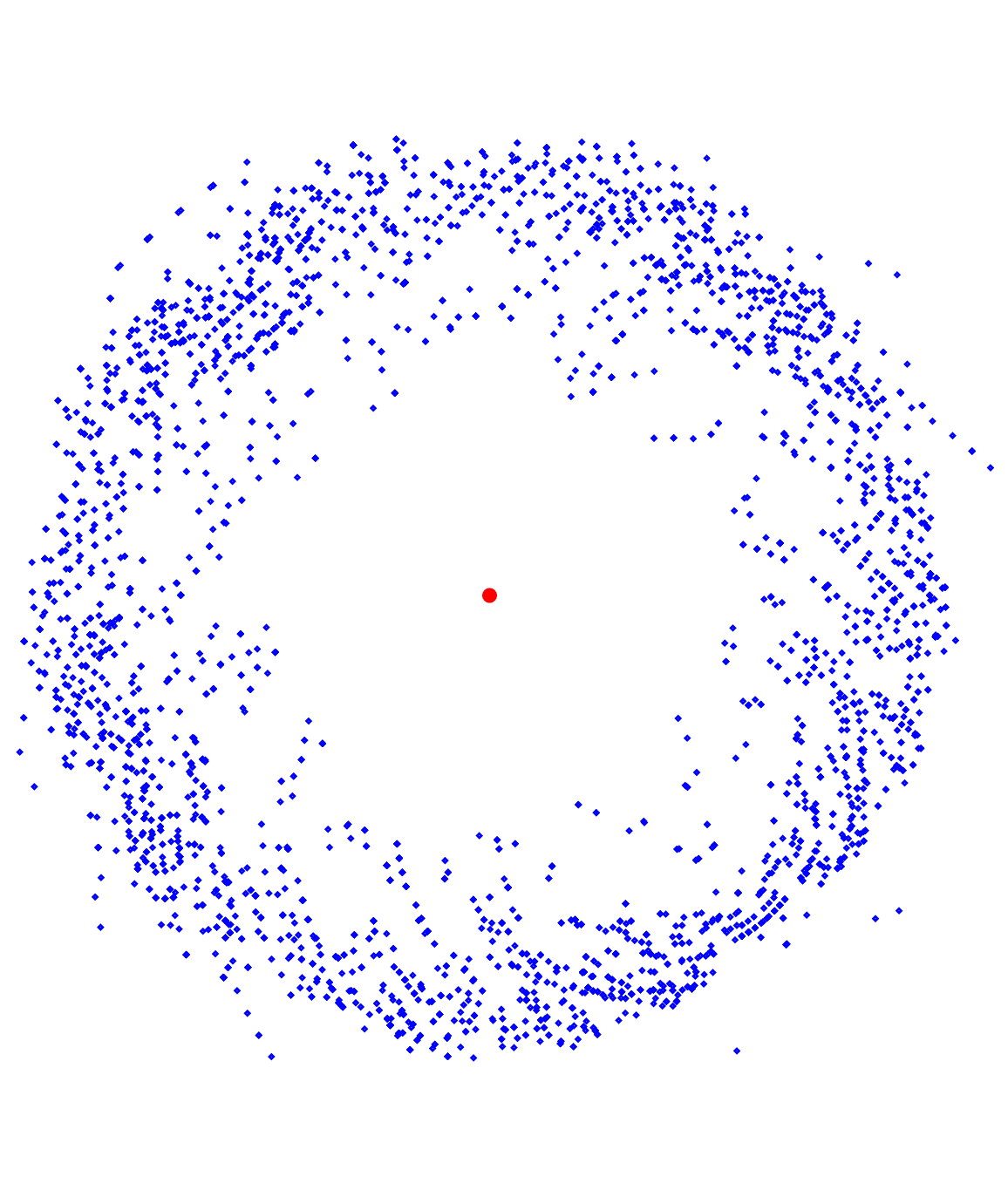}
         \caption{Generate a series of points converges to the solution}
         \label{fig1-2}
     \end{subfigure}
     \caption{Comparison of reflective Kaczmarz algorithm (\ref{standard verison}) for consistent and inconsistent linear systems. The red point in the center is the optimal solution. The blue points are generated by the procedure (\ref{standard verison}).}
     \label{fig1}
\end{figure}

\begin{figure}
     \centering
     \includegraphics[width=\textwidth]{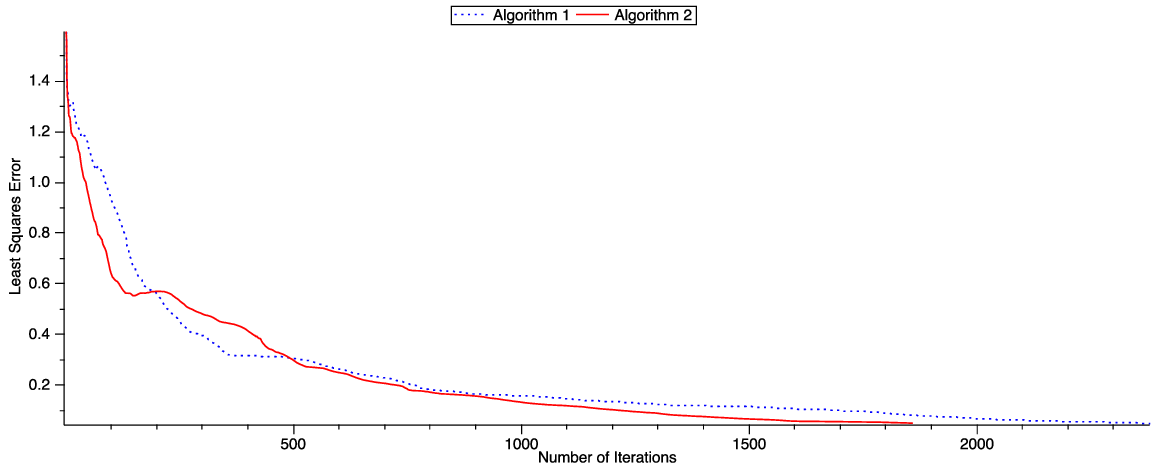}
     \caption{Comparison of {\bf Algorithm 1} and {\bf Algorithm 2}.}
     \label{fig2}
\end{figure}

\begin{figure}
     \centering
     \includegraphics[width=\textwidth]{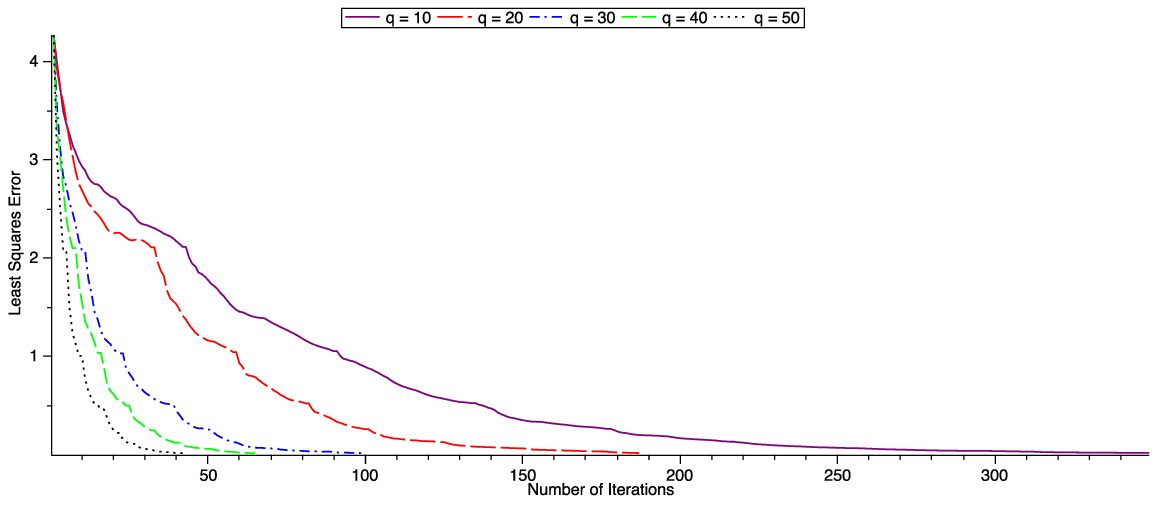}
     \caption{Comparison of {\bf Algorithm 2} for different choices of $q$ for random linear systems of size $300\times 100$.}
     \label{fig3}
\end{figure}

\begin{figure}
     \centering
     \includegraphics[width=\textwidth]{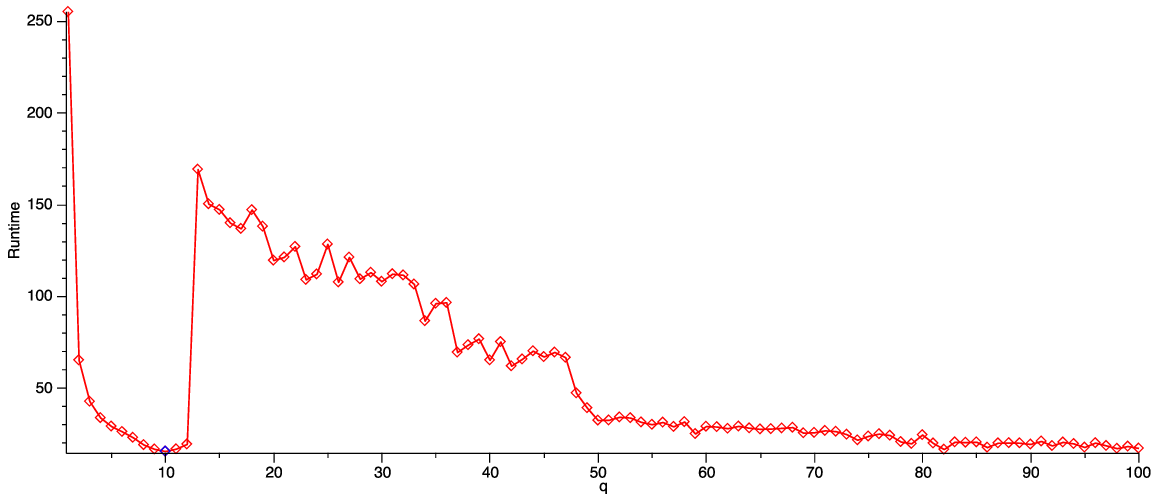}
     \caption{The runtime (in seconds) of {\bf Algorithm 2} for different choices of $q\in\{1,2,\ldots,100\}$ for random linear systems of size $300\times 100$. }
     \label{fig4}
\end{figure}

\begin{figure}
     \centering
     \includegraphics[width=0.5\textwidth]{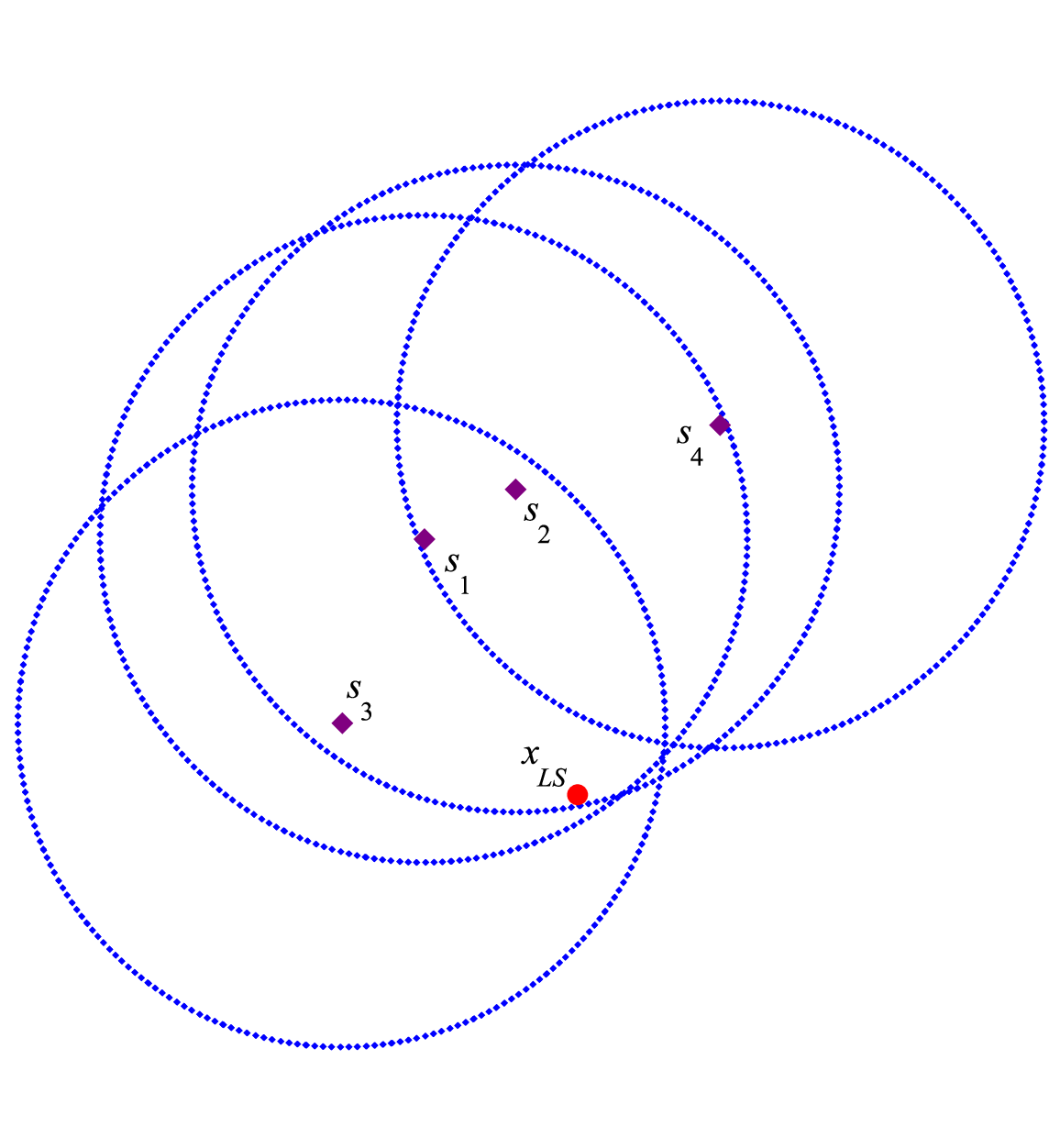}
     \caption{A demonstration of $\x_*^{(i)}$, which is denoted as $s_i$ because of some reason of exhibition, and $\x_{\rm LS}$ in dimension 2.}
     \label{fig0}
\end{figure}

In Figure \ref{fig2}, we show the difference between {\bf Algorithm 1} and {\bf Algorithm 2}. The theoretical results (i.e., Theorems \ref{thm:alg1-consistent-case}, \ref{thm:alg2-consistent-case}, \ref{thm:alg1-inconsistent-case} and \ref{thm:alg2-inconsistent-case}) indicate that the difference is small. This is also verified by numerical tests on random linear systems of size  $300\times 100$. In the plot, the least squares error is given by $\|\x_*-\frac{1}{N}\sum_{k=0}^{N-1}\x_k\|$. We stopped the algorithm when this error is smaller than $0.01$. This explains why the two lines in Figure \ref{fig2} stop at different numbers of iterations.
In comparison, {\bf Algorithm 2} offers more flexibility in implementation, as there is no need to fix the partition of $A$ in advance.

In Figure \ref{fig3}, we demonstrate the convergence rate of {\bf Algorithm 2} for different choices of $q$. As shown in the figure, to obtain a desired least squares error, the number of iterations reduces when $q$ increases. This result aligns with our theoretical expectations, as stated in Theorem \ref{thm:alg2-consistent-case}.

Figure \ref{fig4} shows the overall running time of {\bf Algorithm 2} for different choices of $q$. The minimum runtime is achieved when $q=10$. The runtime decreases initially as $q$ increases, but there is a sudden increase when $q=13$. After that the runtime shows a tendency to decrease. All tested examples of size $300\times 100$ exhibit similar behavior, as shown in Figure \ref{fig4}. This phenomenon is reasonable because, theoretically, the algorithm converges faster when more rows are used. At the same time, the runtime of computing pseudoinverse also increases. So there is a trade-off between the number of rows used at each step and the runtime for computing the pseudoinverse.

We  emphasize again that, in practice, to run {\bf Algorithm 1} and {\bf Algorithm 2}, we would run the algorithms for a while to obtain a series of points $\{\x_0, \x_1, \ldots, \x_{N-1}\}$ and then restart the algorithm with this average as new initial point. Figures \ref{fig2}, \ref{fig3} and \ref{fig4} were indeed obtained in this way.

Figure \ref{fig0} illustrates $\x_*^{(i)}$ and $\x_{\rm LS}$. Each $\x_*^{(i)}$ is the center of a sphere and $\x_*^{(i)}$ is obtained from $\x_*^{(i-1)}$ via a reflection and a translation by Proposition \ref{prop:relation of solutions}. In Corollary \ref{cor:connection} and Proposition \ref{prop:connection},  we provided some connections between $\x_*^{(i)}$ and $\x_{\rm LS}$. But from Figure \ref{fig0}, it is hard to imagine how they are connected.

Finally, we make a remark based on our numerical tests. Although the theoretical results on the convergence rate of block Kaczmarz methods presented in this paper are similar to those of many previous studies, our numerical experiments reveal that the restarting strategy causes (block) reflective Kaczmarz methods to converge much faster than expected. This might be an advantage of this algorithm over others.

\section*{Acknowledgements}

The research is supported by the National Key Research Project of China under Grant No. 2023YFA1009403. 

\bibliographystyle{siam}
\bibliography{main}

\begin{thebibliography}{10}

\bibitem{angelos1992triangular}
{\sc J.~R. Angelos, C.~C. Cowen, and S.~K. Narayan}, {\em {Triangular truncation and finding the norm of a Hadamard multiplier}}, Linear Algebra and its Applications, 170 (1992), pp.~117--135.

\bibitem{benzi2004gianfranco}
{\sc M.~Benzi}, {\em {Gianfranco Cimmino’s contributions to numerical mathematics}}, Atti del Seminario di Analisi Matematica, Dipartimento di Matematica dell’Universita di Bologna. Volume Speciale: Ciclo di Conferenze in Ricordo di Gianfranco Cimmino,  (2004), pp.~87--109.

\bibitem{candes2005decoding}
{\sc E.~J. Candes and T.~Tao}, {\em Decoding by linear programming}, IEEE transactions on Information Theory, 51 (2005), pp.~4203--4215.

\bibitem{Cimmino}
{\sc G.~Cimmino}, {\em Calcolo approssimato per le soluzioni dei sistemi di equazioni lineari}, La Ricerca Scientifica, II (1938), pp.~326--333.

\bibitem{coxeter1939osung}
{\sc H.~M. Coxeter}, {\em L{\"o}sung der aufgabe 245}, Jahresbericht der Deutschen Mathematiker Vereinigung, 49 (1939), pp.~4--6.

\bibitem{coxeter1947product}
\leavevmode\vrule height 2pt depth -1.6pt width 23pt, {\em The product of three reflections}, Quarterly of Applied Mathematics, 5 (1947), pp.~217--222.

\bibitem{elfving1980block}
{\sc T.~Elfving}, {\em Block-iterative methods for consistent and inconsistent linear equations}, Numerische Mathematik, 35 (1980), pp.~1--12.

\bibitem{elsner1982perturbation}
{\sc L.~Elsner and J.-g. Sun}, {\em Perturbation theorems for the generalized eigenvalue problem}, Linear Algebra and its Applications, 48 (1982), pp.~341--357.

\bibitem{gower2015randomized}
{\sc R.~M. Gower and P.~Richt{\'a}rik}, {\em Randomized iterative methods for linear systems}, SIAM Journal on Matrix Analysis and Applications, 36 (2015), pp.~1660--1690.

\bibitem{han2024randomized}
{\sc D.~Han, Y.~Su, and J.~Xie}, {\em {Randomized Douglas--Rachford Methods for Linear Systems: Improved Accuracy and Efficiency}}, SIAM Journal on Optimization, 34 (2024), pp.~1045--1070.

\bibitem{karczmarz1937angenaherte}
{\sc S.~Karczmarz}, {\em Angenaherte auflosung von systemen linearer glei-chungen}, Bull. Int. Acad. Pol. Sic. Let., Cl. Sci. Math. Nat.,  (1937), pp.~355--357.

\bibitem{liu2016accelerated}
{\sc J.~Liu and S.~Wright}, {\em {An accelerated randomized Kaczmarz algorithm}}, Mathematics of Computation, 85 (2016), pp.~153--178.

\bibitem{ma2015convergence}
{\sc A.~Ma, D.~Needell, and A.~Ramdas}, {\em {Convergence properties of the randomized extended Gauss--Seidel and Kaczmarz methods}}, SIAM Journal on Matrix Analysis and Applications, 36 (2015), pp.~1590--1604.

\bibitem{moorman2021randomized}
{\sc J.~D. Moorman, T.~K. Tu, D.~Molitor, and D.~Needell}, {\em {Randomized Kaczmarz with averaging}}, BIT Numerical Mathematics, 61 (2021), pp.~337--359.

\bibitem{necoara2019faster}
{\sc I.~Necoara}, {\em {Faster randomized block Kaczmarz algorithms}}, SIAM Journal on Matrix Analysis and Applications, 40 (2019), pp.~1425--1452.

\bibitem{needell2010randomized}
{\sc D.~Needell}, {\em {Randomized Kaczmarz solver for noisy linear systems}}, BIT Numerical Mathematics, 50 (2010), pp.~395--403.

\bibitem{needell2014paved}
{\sc D.~Needell and J.~A. Tropp}, {\em {Paved with good intentions: analysis of a randomized block Kaczmarz method}}, Linear Algebra and its Applications, 441 (2014), pp.~199--221.

\bibitem{needell2014stochastic}
{\sc D.~Needell, R.~Ward, and N.~Srebro}, {\em {Stochastic gradient descent, weighted sampling, and the randomized Kaczmarz algorithm}}, Advances in Neural Information Processing Systems, 27 (2014).

\bibitem{needell2015randomized}
{\sc D.~Needell, R.~Zhao, and A.~Zouzias}, {\em {Randomized block Kaczmarz method with projection for solving least squares}}, Linear Algebra and its Applications, 484 (2015), pp.~322--343.

\bibitem{schreiber1989storage}
{\sc R.~Schreiber and C.~Van~Loan}, {\em {A storage-efficient WY representation for products of Householder transformations}}, SIAM Journal on Scientific and Statistical Computing, 10 (1989), pp.~53--57.

\bibitem{shao2021deterministic}
{\sc C.~Shao}, {\em {A deterministic Kaczmarz algorithm for solving linear systems}}, SIAM Journal on Matrix Analysis and Applications, 44 (2023), pp.~212--239.

\bibitem{steinerberger2021surrounding}
{\sc S.~Steinerberger}, {\em Surrounding the solution of a linear system of equations from all sides}, Quarterly of Applied Mathematics, 79 (2021), pp.~419--429.

\bibitem{stewart2004elsner}
{\sc G.~W. Stewart}, {\em {An Elsner-like perturbation theorem for generalized eigenvalues}}, Linear Algebra and its Applications, 390 (2004), pp.~1--5.

\bibitem{stewart1990matrix}
{\sc G.~W. Stewart and J.-g. Sun}, {\em {Matrix Perturbation Theory}}, Academic Press, INC, 1990.

\bibitem{strohmer2009randomized}
{\sc T.~Strohmer and R.~Vershynin}, {\em {A randomized Kaczmarz algorithm with exponential convergence}}, Journal of Fourier Analysis and Applications, 15 (2009), pp.~262--278.

\bibitem{woodruff2014sketching}
{\sc D.~P. Woodruff et~al.}, {\em Sketching as a tool for numerical linear algebra}, Foundations and Trends{\textregistered} in Theoretical Computer Science, 10 (2014), pp.~1--157.

\bibitem{zouzias2013randomized}
{\sc A.~Zouzias and N.~M. Freris}, {\em {Randomized extended Kaczmarz for solving least squares}}, SIAM Journal on Matrix Analysis and Applications, 34 (2013), pp.~773--793.

\end{thebibliography}

\end{document}